\documentclass[11pt, oneside]{amsart}

\usepackage{pinlabel}

\usepackage{amsmath, amsfonts, amssymb}
\usepackage{amsopn}
\usepackage{enumerate}
\usepackage{ifthen, color, hyperref}
\usepackage[cmtip,arrow]{xy}

\usepackage{enumitem}

\usepackage[T1]{fontenc}

\newcommand{\showcomments}{yes}

\newsavebox{\commentbox}
\newenvironment{com}%
{\ifthenelse{\equal{\showcomments}{yes}}%
{\footnotemark
        \begin{lrbox}{\commentbox}
        \begin{minipage}[t]{1.in}\raggedright\sffamily\tiny
        \footnotemark[\arabic{footnote}]}
{\begin{lrbox}{\commentbox}}}%
{\ifthenelse{\equal{\showcomments}{yes}}%
{\end{minipage}\end{lrbox}\marginpar{\usebox{\commentbox}}}
{\end{lrbox}}}

\newcounter{ax}
\newtheorem{thm}{Theorem}[section]

\newtheorem{lem}[thm]{Lemma}
\newtheorem{prop}[thm]{Proposition}

\newtheorem{cor}[thm]{Corollary}

\theoremstyle{definition}
\newtheorem{defn}[thm]{Definition}
\newtheorem{rem}[thm]{Remark}

\newtheorem{claim}{Claim}

\newtheorem{claim*}{Claim}

\makeatletter

\newcommand{\Rmnum}[1]{\mathbf{{\expandafter\@slowromancap\romannumeral #1@}}}

 \makeatother

\let\oldmarginpar\marginpar
\renewcommand\marginpar[1]{\-\oldmarginpar[\raggedleft\footnotesize #1]%
{\raggedright\footnotesize #1}}

\newcommand\ds{\displaystyle}

\DeclareMathOperator{\supp}{\mathrm{supp}}

\newcounter{enumitemp}

\def\G{{\mathcal G}}
\newcommand{\CFS}{\ensuremath{\mathcal{CFS}}}

\newcommand{\E}{\mathbb E}
\newcommand{\Pb}{\mathbb P}

\newcommand{\edge}[2]{{#1}{#2}}

\newcommand{\pair}[2]{{#1}{#2}}
\newcommand{\pairgraph}[1]{\square\left({#1}\right)}

\newcommand{\noninduced}{\boxtimes(\Gamma)}

\definecolor{figblue}{cmyk}{1,0,0,0}
\definecolor{figred}{cmyk}{0,1,1,0}
\definecolor{figgreen}{cmyk}{.75,0,1,0}
\definecolor{figblue2}{cmyk}{1,1,0,0}
\definecolor{figgreen2}{cmyk}{1,0,1,0}

\title[Square percolation and the threshold for quadratic 
divergence in random RACG]{Square percolation and the threshold for quadratic 
divergence in random right-angled Coxeter groups}
\author[J. Behrstock]{Jason Behrstock}
\address{Lehman College and The Graduate Center, CUNY, New York, New York, USA}
\email{jason.behrstock@lehman.cuny.edu}

\author[V. Falgas-Ravry]{Victor Falgas-Ravry}
\address{Ume{\aa} Universitet, Ume{\aa}, Sweden}
\email{victor.falgas-ravry@umu.se}

\author[T. Susse]{Tim Susse}
\address{Bard College at Simon's Rock, Great Barrington, Mass., USA}
\email{tsusse@simons-rock.edu}


\date{\today}

\numberwithin{thm}{section}
\numberwithin{equation}{section}
\begin{document}
		\begin{abstract}
		Given a graph $\Gamma$, its auxiliary \emph{square-graph}
		$\square(\Gamma)$ is the graph whose vertices are the
		non-edges of $\Gamma$ and whose edges are the pairs of
		non-edges which induce a square (i.e., a $4$-cycle) in
		$\Gamma$.  We determine the threshold edge-probability
		$p=p_c(n)$ at which the Erd{\H o}s--R\'enyi random graph
		$\Gamma=\Gamma_{n,p}$ begins to asymptotically almost surely
		have a square-graph with a connected component whose squares
		together cover all the vertices of $\Gamma_{n,p}$.  We show
		$p_c(n)=\sqrt{\sqrt{6}-2}/\sqrt{n}$, a polylogarithmic
		improvement on earlier bounds on $p_c(n)$ due to Hagen and the
		authors.  As a corollary, we determine the threshold
		$p=p_c(n)$ at which the random right-angled Coxeter group
		$W_{\Gamma_{n,p}}$ asymptotically almost surely becomes
		strongly algebraically thick of order $1$ and has quadratic
		divergence.
	\end{abstract}

	\maketitle

\section{Introduction}
In this paper we investigate the phase transition for a variant of
``square percolation'', with motivation coming from both previous work
on clique percolation and from questions in geometric group theory.

Clique percolation was introduced by Der\'enyi, Palla and
Vicsek \cite{DerenyiPallaVicsek:CliquePercolation} as a simple model
for community detection, and quickly became well-studied in network
science, from computational, empirical, and theoretical perspectives, see e.g.~\cite{BollobasRiordan:cliquepercolation,
	DerenyiPallaVicsek:CliquePercolation,
	LiDengWang:cliqueperc, DerenyiPallaVicsek:CommunityStructure,
TothVicsekPalla:overlappingmodularitycliqueperc,WangCaoSuzukiAihara:Epidemicscliqueperc}.
In $(k, \ell)$--clique percolation, to investigate the ``community 
structure'' of a graph or network
$\Gamma$ one studies the auxiliary $(k, \ell)$-clique graph whose vertices are the
$k$--cliques of $\Gamma$ and whose edges are those pairs of
$k$--cliques having at least $\ell$ vertices in common.

One of the main research questions in the area was determining the
threshold $p=p(n)$ for the emergence of a ``giant component'' in the
auxiliary $(k, \ell)$-clique graph when the original graph $\Gamma\in
\mathcal{G}(n,p)$ is an Erd{\H o}s--R\'enyi random graph on $n$
vertices with edge-probability $p$.  This was completely resolved in
2009 by Bollob\'as and
Riordan~\cite{BollobasRiordan:cliquepercolation}, in a highly
impressive paper making sophisticated use of branching processes.  In
the concluding remarks of their paper, Bollob\'as and Riordan
suggested a study of ``square percolation'' as a natural extension of
their work.  More precisely, given a graph $\Gamma$ they suggested
studying the component structure of the auxiliary graph whose vertices
are the \emph{not necessarily induced} $4$-cycles in $\Gamma$, and
whose edges are pairs of $4$--cycles with a diagonal\footnote{Note
that given a $4$--cycle in a graph, we use the term ``diagonal'' to
refer to the pair of vertices of a diagonal, even though they they may
not span an edge in the graph; indeed most of this paper concerns
induced $4$--cycles so that the edge spanned by the diagonal is
\emph{not} in the graph.} in common.  For 
$\Gamma\in \mathcal{G}(n,p)$, they stated that they believed the
threshold for the associated 
auxiliary graph to contain a giant component 
containing a positive
proportion of all squares of $\Gamma$ should be $\lambda_c/\sqrt{n}$,
where $\lambda_c=\sqrt{\sqrt{6}-2}$ (see the discussion around
equation (19) in Section~2.3
of~\cite{BollobasRiordan:cliquepercolation}).

A related (but slightly different) notion of ``square percolation'' arose
independently in joint work of the authors with Hagen
\cite{BehrstockFalgasRavryHagenSusse:StructureRandom} on the divergence of
the random right-angled Coxeter group, providing motivation from
geometric group theory for understanding the phase transition in an auxiliary graph formed from the \emph{induced} $4$-cycles of an Erd{\H o}s--R\'enyi random graph $\Gamma\in \mathcal{G}(n,p)$. To make this more precise, we make the following definition.
\begin{defn}\label{def: square graph}
To any graph $\Gamma$, we associate an auxiliary
\emph{square-graph}, $\pairgraph{\Gamma}$, whose vertices are the non-edges of $\Gamma$, 
and whose edges are the pairs of non-edges of $\Gamma$ that together 
induce a $4$--cycle (a.k.a.\ \emph{square}) in $\Gamma$.	
\end{defn}
Thus for vertices $a,b,c,d$ in a graph $\Gamma$, the pair $\{ac,
bd\}$ is an edge of $\square (\Gamma)$ if and only if (i) $ac$ and
$bd$ are non-edges of $\Gamma$ (and thus vertices of
$\square(\Gamma)$), and (ii) $ab$, $bc$, $cd$ and $da$ are all edges
of $\Gamma$.  
\begin{rem} This definition of the auxiliary
	square-graph $\square(\Gamma)$ differs slightly from
	the one used in the related papers
	\cite{BehrstockFalgasRavryHagenSusse:StructureRandom,
		DaniThomas:divcox}. In those papers, the auxiliary graph had the
	induced $4$-cycles as its vertices, and its edges were those pairs of
	induced $4$-cycles having a diagonal in common.  These two variants of auxiliary square-graphs encode essentially the same information, but the formulation above is more natural from a
	combinatorial perspective and more convenient for the exploration
	processes we shall consider in this paper.
\end{rem}
We investigate the component structure of $\square(\Gamma)$, albeit 
with an unusual twist. With a view to applications in geometric group 
theory, we will be interested in the question of whether or not 
$\square(\Gamma)$ has a component that ``covers'' all of the vertex-set of the original graph $\Gamma$.
\begin{defn}\label{def: support of a square component}
	We refer to connected components of $\pairgraph{\Gamma}$ as
	\emph{square-components} of $\Gamma$.  Given a square-component $C$ we define
	its \emph{support} to be the collection of vertices of $\Gamma$ given 
	by: 
	\[\supp(C)=\bigcup_{vw\in C}\{v,w\},\]and 
	say that the component $C$ of $\pairgraph{\Gamma}$ 
	\emph{covers} the vertex set $\supp(C)\subseteq V(\Gamma)$. If $C$ covers all of $V(\Gamma)$, we say it is a square-component with \emph{full support}
\end{defn} 
Write $\Gamma\in \mathcal{G}(n,p)$ to denote that $\Gamma$ is an instance
of the Erd{\H o}s--R\'enyi random graph model with parameters $n$ and
$p$, i.e., that $\Gamma$ is a graph on $n$ vertices obtained by including each edge at
random with probability $p$, independently of all the others.

Our main combinatorial result in this paper is pinpointing the precise
threshold $p_c(n)$ at which $\Gamma\in \mathcal{G}(n,p)$
asymptotically almost surely\footnote{As usual, \emph{asymptotically
almost surely} or \emph{a.a.s.} is shorthand for ``with probability
tending to $1$ as $n\rightarrow \infty$.''} experiences a phase
transition from having only square-components with support of
logarithmic order to having a square-component with full support.
Throughout this paper we set $\lambda_c=\sqrt{\sqrt{6}-2}$.  The
following two results establish that the critical threshold 
probability is 
$p(n)=\lambda_cn^{-1/2}$ by showing highly disparate behavior 
on either side of this threshold as given by the following two 
contrasting results.

\begin{thm}[Subcritical Behavior] \label{thm:no_giant} 
	Let $\lambda<\lambda_c$ be fixed. Suppose that $p(n)\le \lambda n^{-1/2}$.
	Then for $\Gamma\in\G(n,p)$, a.a.s.\ 
every square-component of $\Gamma$ covers at most $O((\log n)^{2^{32}})$ vertices.
\end{thm}

\begin{thm}
	[Supercritical Behavior] \label{thm:1giant}
Let $\lambda>\lambda_c$ be fixed, and let $f\colon
\mathbb{N}\rightarrow \mathbb{R}^+$ be a function with
$f(n)\rightarrow 0$ and $f(n)n^2\rightarrow \infty$ as $n\rightarrow
\infty$.  Let $p=p(n)$ be an edge-probability with
	\[	\lambda n^{-1/2} \leq  p(n) \leq 1-f(n).\]
	Then for $\Gamma\in\G(n,p)$, a.a.s.\ there is a square-component of
	$\Gamma$
	covering \emph{all} vertices of $\Gamma$. 		
\end{thm}

Our proofs of Theorems~\ref{thm:no_giant}
and~\ref{thm:1giant} confirm the conjecture of Bollob\'as and Riordan
regarding the location of the phase transition for their version of
(non-induced) square percolation.  Further,
Theorems~\ref{thm:no_giant} and~\ref{thm:1giant} have a direct
application to the study of the geometric properties of random
right-angled Coxeter group, which we now describe.

Given a graph $\Gamma$, we define the associated \emph{right-angled
Coxeter group} (RACG) $W_{\Gamma}$ by taking the free group on
$V(\Gamma)$ and adding the relations $a^2=1$ and $ab=ba$ for all $a\in
V(\Gamma)$, $ab\in E(\Gamma)$.  In this way, the graph $\Gamma$
encodes a \emph{finite presentation} for the right-angled Coxeter
group $W_\Gamma$.  Given graphs $\Gamma$ and $\Lambda$, it is 
well-known that the associated
groups $W_\Gamma$ and $W_\Lambda$ are isomorphic if and only if the
graphs $\Gamma$ and $\Lambda$ are isomorphic, see
\cite{Muhlherr:aut_cox}.
Thus algebraic and geometric properties of $W_{\Gamma}$ can be studied
via purely graph-theoretic means, as we do in this paper.  Indeed, a
number of geometric properties of a right-angled Coxeter
group $W_{\Gamma}$ admit encodings as graph-theoretic properties of
the presentation graph $\Gamma$.  Such properties include thickness and having quadratic
divergence, which are both important in geometric group
theory (see Section~\ref{sec:background} below for a formal
definition of these notions). An investigation of  
right-angled Coxeter groups with 
quadratic divergence was the main motivation for
the work undertaken in this paper.

The correspondence between right-angled Coxeter groups and graphs
allows one to define models of
random groups based on random graph models. In particular, in this 
paper we consider 
the \emph{random right-angled Coxeter group},
$W_{\Gamma}$ where the presentation graph $\Gamma\in \mathcal{G}_{n,p}$ is an Erd{\H o}s--R\'enyi random graph. 
Using Theorems~\ref{thm:no_giant} and \ref{thm:1giant} above on square-components in Erd{\H
o}s--R\'enyi random graphs, we prove the following. 
	\begin{thm}[Criticality for quadratic divergence of
	RACGs]\label{thm: threshold for quadratic divergence}Let
	$\epsilon>0$.  If
	\[	\frac{\lambda_c+\epsilon }{\sqrt{n}} \leq p(n) \leq 1- \frac{(1+\epsilon)\log{n}}{n}\]
	and $\Gamma\in \G(n,p)$. Then, a.a.s.\ the right-angled
	Coxeter group $W_\Gamma$  has quadratic divergence and is strongly algebraically thick of order exactly $1$. 
	
	On the other hand, if $p(n)$ satisfies
	\[0\leq p(n)\leq \frac{\lambda-\epsilon }{\sqrt{n}} \] then the
	right-angled Coxeter group $W_\Gamma$ a.a.s has at least cubic
	divergence and is not strongly algebraically thick of order $0$ or $1$.
\end{thm}

The geometric properties of $W_{\Gamma}$ when $\Gamma\in \G_{n,p}$ and
$p=1-\theta(n^{-2})$ was described in detail by Behrstock, Hagen and
Sisto in \cite[Theorem V]{BehrstockHagenSisto:coxeter}.  Together with
their work, our results give an essentially complete picture of
quadratic and linear divergence in random right-angled Coxeter groups.

\subsection*{Organization of the paper}
In Section~\ref{sec:background} we provide additional background
material on the geometry of random groups and derive Theorem~\ref{thm: threshold for quadratic divergence} from Theorems~\ref{thm:no_giant}--\ref{thm:1giant}.  In Section~\ref{section: strategy}, we recall some basic facts about branching processes and give an outline of the proof
strategy we follow for our main results, and of the ways in which it departs from the framework used by Bollo\'as and Riordan in their study of clique percolation in random graphs.
Theorem~\ref{thm:no_giant} is proved in Section~\ref{section: lower bound}, while Theorem~\ref{thm:1giant}  is derived in Section~\ref{section: upper bound}. We 
end the paper in Section~\ref{section: conclusion} with some discussion of the results and of further work and related problems.

\subsection*{Acknowledgments} The authors thank Mark Hagen for 
discussions during the early stages of this work. 
Behrstock was supported by NSF grant DMS-1710890. 
Falgas-Ravry was supported by Swedish Research Council grant  
VR 2016-03488. The authors thank Ela Behrstock for her skillful 
rendering of Figures~\ref{fig:Figure_Process_1} and~\ref{fig:Figure_Process_2}.
Aspects of this work 
were motivated by output from software written by the authors 
(available upon request from the authors, some available online at 
\url{http://comet.lehman.cuny.edu/behrstock/random.html}). 
Accordingly, this research was supported, in part, by a grant of computer time
from the City University of New York High Performance Computing
Center under NSF Grants CNS-0855217, CNS-0958379, and ACI-1126113.

\section{Graph-theoretic notation and standard notions}
Given a set $A$ and $r\in \mathbb{N}$, let $A^{(r)}$ denote the 
collection of all subsets of $A$ of cardinality $r$. So for example $A^{(2)}$ is the collection unordered distinct pairs of elements of $A$. As a notational convenience, we set $[n]:=\{1, 
2, \ldots, n\}$, and we often denote the unordered set 
$\{u,v\}$ by $uv$.

A graph is a pair $\Gamma=(V,E)$, where $V=V(\Gamma)$ is a set of
\emph{vertices} and $E=E(\Gamma)$ is a collection of pairs of vertices
referred to as the \emph{edges} of $\Gamma$.  A subgraph of $\Gamma$
is a graph $G$ with $V(G)\subseteq V(\Gamma)$ and $E(G)\subseteq
E(\Gamma)$.  If $V(G)=X$ and $E(G)=V(\Gamma)\cap X^{(2)}$, then we say
$G$ is the subgraph of $\Gamma$ induced by $X$ and denote this fact by
$G=\Gamma[X]$. When there is no risk of confusion, we may abuse
notation and use $X$ to refer to both the subset of
$V(\Gamma)$ and the associated induced subgraph $\Gamma[X]$. The complement of a graph $\Gamma=(V, E)$ is the graph $\Gamma^c=(V, V^{(2)}\setminus E)$.

A path of length $\ell$ in a graph $\Gamma$ is an ordered sequence 
of  $\ell+1$ distinct vertices $v_0, v_1, \ldots, v_{\ell}$ together 
with a set of $\ell$ edges $\{v_{i-1}v_i: \ i \in [l]\}\subseteq 
E(\Gamma)$. Such a path is said to join $v_0$ to $v_{\ell}$. Two 
vertices are said to be connected in $\Gamma$ if there is a path joining them.  Being connected is an equivalence relation on the vertices of $\Gamma$. A \emph{(connected) component} of $\Gamma$ is then a nonempty set of vertices from $V(\Gamma)$ that forms an equivalence class under this relation.

In this paper we study
\emph{squares} in graphs.  An square, or $4$--cycle, in $\Gamma$ is a copy of the graph $C_4=(\{a,b,c,d\}, \{ab, bc, cd, da\})$ as an subgraph of $\Gamma$.  In an abuse of notation, we
will denote such a $C_4$ by $abcd$.  In other words, if we say
``$abcd$ is a copy of $C_4$/a square in $\Gamma$'', we mean ``$ab, bc,
cd, da\in E(\Gamma)$''.  Further if we say ``$abcd$ is an induced
$C_4$/square in $\Gamma$'', we mean that $abcd$ is a square in
$\Gamma$ and that in addition $ac, bd\notin E(\Gamma)$.  A useful
notion for studying squares in graphs is that of a link
graph: given a vertex $x\in V(\Gamma)$, the \emph{link graph}
$\Gamma_x$ of $x$ is the collection of \emph{neighbors} of $x$ in
$\Gamma$, i.e., $\Gamma_x=\{y\in V(\Gamma):\ xy\in E(\Gamma)\}$.

By $\Gamma\in\G(n,p)$ we mean that $\Gamma$ is a random graph on the
vertex set $[n]$ obtained by including each edge $uv$ in $E(\Gamma)$
with probability $p$, independently of all the others.  This is known
as the Erd{\H o}s--R\'enyi random graph model.  Given a sequence of
edge probabilities $p=p(n)$ and a graph property $\mathcal{P}$, we say
that a typical instance of $\Gamma\in\G(n,p)$ has property
$\mathcal{P}$, or, equivalently, that $\Gamma \in \mathcal{P}$ holds
\emph{asymptotically almost surely (a.a.s.)} if
\[\lim_{n\rightarrow \infty} \mathbb{P}(\Gamma \in \mathcal{P})=1.\]
Throughout the paper, we use standard Landau notation: given 
functions $f,g\colon \mathbb{N} \rightarrow \mathbb{R}^+$, we write $f=o(g)$ for $\lim_{n\rightarrow \infty} f(n)/g(n)=0$ and $f=O(g)$ if there exists a constant $C>0$ such that $\limsup_{n\rightarrow \infty}f(n)/g(n)\leq C$. Further we write $f=\omega(g)$ for $g=o(f)$, $f=\Omega(g)$ for $g=O(f)$. Finally if $f=O(g)$ and $f=\Omega(g)$ both hold, we denote this fact by $f=\Theta(g)$.

\section{Geometric group theory and the $\CFS$ property}\label{sec:background}

 Our main result in this paper establishes that
 $p(n)=\lambda_c/\sqrt{n}$ is the threshold for a typical instance
 $\Gamma$ of the Erd{\H o}s--R\'enyi random graph model
 $\mathcal{G}(n,p)$ to have a square-graph with a component covering
 all of $V(\Gamma)$.  This property is a.a.s.\ equivalent to
 possessing the $\mathcal{CFS}$--property, defined below.  
 
 \subsection{Background}
 
 Recall that
 the graph joint $\Gamma_1\ast \Gamma_2$ of two graphs $\Gamma_1$ and
 $\Gamma_2$ is the graph obtained by taking disjoint unions of
 $\Gamma_1$ and $\Gamma_2$, and adding in all edges from $\Gamma_1$ to
 $\Gamma_2$.
 \begin{defn}
A finite graph $\Gamma$ is defined to be $\CFS$ (\emph{constructed from
squares}) if $\Gamma$ has induced subgraphs $K$ and $\Gamma'$ with $K$
a (possibly empty) clique so that:
 \begin{itemize}
 	\item $\Gamma=\Gamma'\ast K$, and 
 	\item $\square(\Gamma')$ has a component $C$ with $\supp(C)=V(\Gamma')$.
 \end{itemize}
 \end{defn}
 Dani--Thomas were the first to introduce a special case of the $\CFS$
 property for triangle-free graphs in~\cite{DaniThomas:divcox}. The
 $\CFS$ property for arbitrary graphs was then studied by Hagen and the
 authors in~\cite{BehrstockFalgasRavryHagenSusse:StructureRandom}, 
 with an eye towards establishing when this property holds a.a.s.\ in random graphs, while in \cite{Levcovitz:CFSdiv} Levcovitz studied the geometric properties of right-angled Coxeter  groups whose presentation graphs do not possess the $\CFS$ property.

With Hagen, the authors determined in~ \cite{BehrstockFalgasRavryHagenSusse:StructureRandom} the threshold for the $\CFS$ property to hold a.a.s.\ in Erd{\H o}s--R\'enyi random graphs up to a polylogarithmic factor. 
 \begin{thm}[Theorems 5.1 and 5.7 in~\cite{BehrstockFalgasRavryHagenSusse:StructureRandom}]\label{thm: old results}
	  If $p(n)\leq \left(\log n\right)^{-1}/\sqrt{n}$, then a.a.s.\ a
	 graph $\Gamma\in\G(n,p)$ does not have the $\CFS$ property.  On
	 the other hand if $p(n)\geq 5\sqrt{\log n} /\sqrt{n}$ and $(1-p)n^2\to
	 \infty$, then a.a.s.\ $\Gamma\in\G(n,p)$ does have the $\CFS$
	 property.
\end{thm}
\noindent Our contribution in this paper is to eliminate the polylogarithmic gap in Theorem~\ref{thm: old results} and thus to determine the precise threshold for the $\CFS$ property in random graphs.

The $\CFS$ property is closely linked to the large scale geometry of
right-angled Coxeter groups, connected to \emph{divergence} and
\emph{(strong algebraic) thickness}.  Divergence is a quasi-isometry
invariant of groups introduced by Gersten~\cite{Gersten:divergence}
and further developed by Dru\c{t}u, Mozes and Sapir
\cite{DrutuMozesSapir}, while thickness was introduced by
Behrstock--Dru\c{t}u--Mosher in \cite{BDM} and then further refined 
by Behrstock--Dru\c{t}u in \cite{BehrstockDrutu:thick2}.  We define these notions
and explain how they are related below.

\begin{defn} Let $(X,d)$ be a geodesic metric space, let $o\in X$ and
let $\rho\in(0,1]$.  Given $x,y\in X$ with $d(x,o)=d(y,o)=r$, we
define $d_{\rho r}(x,y)$ to be the infimum of the lengths of paths in
$X\setminus B(o, \rho r)$ between $x$ and $y$, if such a path exists,
and $\infty$ otherwise; here $B(o, \rho r)$ denotes the ball of 
radius $\rho r$ about $o$.  We then set
	$$\delta_\rho(r)=\sup_{o\in X}\sup_{x,y} d_{\rho r}(x,y).$$
The \emph{divergence} of $X$ is defined to be the collection of
functions $\delta_{\rho}: \ r\mapsto\delta_\rho(r)$,
$$\text{Div}(X):=\{\delta_\rho: \rho\in(0,1]\}.$$
\end{defn}
Given two non-decreasing functions $f,g\colon \mathbb{N}\to\mathbb{R}^+$, we say that $f\lesssim g$ if there exists $C\ge 1$ so that:
$$f(r)\le C\cdot g(Cr+C)+Cr+C,$$ and we say $f\sim g$ if $f\lesssim g$ and $g\lesssim f$. Importantly, two  polynomials that are non decreasing $\mathbb{N}\to\mathbb{R}^+$ and have the same degree are equivalent under this relation, and further for $a,b\in \mathbb{N}$ we have $x^a\sim x^b$ if and only if $a=b$.

When $X$ is the Cayley graph of a right-angled Coxeter group, it is
straightforward to see that $\delta_\rho(r) \sim \delta_1(r)$.
Therefore when we are referring to the \emph{divergence function} of
$W_\Gamma$, we will mean $\delta_1(r)$.  We say that a RACG
$W_{\Gamma}$ has \emph{quadratic divergence} if $\delta_1(r) \sim r^2$
and \emph{linear divergence} if $\delta_1(r)\sim r$.

\begin{defn}\label{def:sathick} Let $G$ be a finitely generated group. 
	\begin{itemize}
		\item We say that $G$ is strongly algebraically thick of order $0$ if it has linear divergence.
		\item We say that $G$ is strongly algebraically thick of order at most $n$ if $G$ has a collection of subgroups $\mathcal H=\{H_\alpha\}$ so that:
		\begin{itemize}
			\item $\left\langle\bigcup_{\alpha} H_\alpha\right\rangle$ has finite index in $G$
			\item for $H_\alpha, H_{\beta}\in \mathcal H$ there exists a sequence $H_0=H_\alpha, H_1, \ldots H_{k-1}, H_k=H_\beta$ of elements of $\mathcal H$ so that $H_{i-1}\cap H_{i}$ is infinite for each $1\le i\le k$
			\item there exists a constant $M>0$ so that each $H_{\alpha}\in\mathcal H$ is \emph{$M$--quasiconvex}, that is to say, every pair of points in $H_{\alpha}$ can be connected by an $(M,M)$--quasigeodesic contained in $H_{\alpha}.$

			\item each $H_\alpha\in \mathcal H$ is strongly algebraically thick of order at most $n-1$.
		\end{itemize}
	\end{itemize}
Further, we say that $G$ is strongly algebraically thick of order
\emph{exactly} $n$ if it is strongly algebraically thick of order at
most $n$ and \emph{not} strongly algebraically thick of order at most
$n-1$.  We also usually write ``thick'' as a shorthand for ``strongly
algebraically thick''.  \end{defn}

Behrstock and Dru\c{t}u discovered that the order of thickness provides upper bounds on the divergence of a metric space. In particular they proved:

\begin{prop}[{\cite[Corollary~4.17]{BehrstockDrutu:thick2}}]\label{prop:thickpolydiv} Let $G$ be a finitely generated group which is strongly algebraically thick of order at most $n$. Then for every $\rho\in(0, 1]$, $\delta_\rho(r)\lesssim r^{n+1}$.\end{prop}

The group theoretic motivation for studying the $\CFS$ property is that it provides a graph theoretical proxy for certain geometric properties of right-angled Coxeter groups, such as their divergence. To see that $W_\Gamma$ has quadratic divergence when $\Gamma$ has the $\CFS$ property is straightforward, since interpreting the  definition of $\CFS$ in the Cayley graph yields a chain of linearly many spaces with linear divergence with each intersecting the next in an infinite diameter set. Indeed, it is an immediate consequence of the definitions that if  if $G$ is the direct product of two infinite groups then $G$ has linear divergence, just as a path avoiding a linear-sized ball in the plane has linear length. 
 Hence every finitely generated abelian group of rank at least $2$ has linear divergence (and is thick of order $0$).

Now, if $\Gamma=K\star \Gamma'$ where $K$ is a clique, then $W_\Gamma\cong \mathbb{Z}_2^{|K|}\times W_{\Gamma'}$. In such a case $W_{\Gamma'}$ is a finite-index subgroup of  $W_{\Gamma'}$ and thus, up to finite index, we can assume that $\Gamma$ does not contain a vertex sending an edge to all other vertices of $\Gamma$.

Now, $W_\Gamma$ contains a network of convex subgroups generated by the induced square in the full-support component of $\square(\Gamma)$. Each of these groups is \emph{virtually $\mathbb{Z}^2$}, that is to say has a finite index subgroup which is a copy of $\mathbb{Z}^2$. Further, two induced squares in $\Gamma$ correspond to incident edges in $\square(\Gamma)$ if and only if the intersection of the associated virtual $\mathbb{Z}^2$ subgroups is \emph{virtually $\mathbb{Z}$}, that is to say has a finite index subgroup which is a copy of $\mathbb{Z}$.

 Thus, paths in the full-support component of $\square(\Gamma)$ give the connecting sequences needed in Definition~\ref{def:sathick}. Hence if $\Gamma$ has the $\CFS$ property, $W_\Gamma$ is thick of order at most $1$ and has at most quadratic divergence.

 As shown by Dani--Thomas in the triangle-free case \cite[Theorem~1.1
 and Remark~4,8]{DaniThomas:divcox}, and by the present authors with
 Hagen in the general case (as above)
 \cite[Proposition~3.1]{BehrstockFalgasRavryHagenSusse:StructureRandom},
 if $\Gamma$ has the $\mathcal{CFS}$ property then the associated
 right-angled Coxeter group $W_\Gamma$ has thickness of order at most
 $1$, and hence has at most quadratic divergence.  Further, in
 \cite{BehrstockHagenSisto:coxeter}, Behrstock, Hagen and Sisto show
 that a right-angled Coxeter group $W_\Gamma$ has linear divergence
 (and is thick of order $0$) if and only if $\Gamma$ is the join of
 two non-complete graphs.  Finally Levcovitz proved that any graph
 without $\CFS$ has at least cubic divergence \cite{Levcovitz:CFSdiv},
 and so we see that $W_\Gamma$ has exactly quadratic divergence if and
 only if $\Gamma$ is not the join of two non-complete graphs and has
 the $\CFS$ property.

 \subsection{Proof of threshold for quadratic divergence in random 
 RACGs}

Assuming our main theorems about square percolation, 
Theorems~\ref{thm:no_giant} and~\ref{thm:1giant}, we are now in a position to 
provide a proof of Theorem~\ref{thm: threshold for quadratic divergence} on the threshold for quadratic divergence in RACGs:
\begin{proof}[Proof of Theorem~\ref{thm: threshold for quadratic divergence} from Theorems~\ref{thm:no_giant} and~\ref{thm:1giant}]

Let $\epsilon>0$, and suppose that
\[\frac{\lambda_c+\epsilon}{\sqrt{n}}\le p(n)\le
1-\frac{(1+\epsilon)\log{n}}{n}.\] By Theorem~\ref{thm:1giant}, the
graph $\Gamma$ a.a.s.\ has the $\CFS$ property.  Thus, by
\cite[Proposition~3.1]{BehrstockFalgasRavryHagenSusse:StructureRandom},
$W_\Gamma$ a.a.s.\ has at most quadratic divergence and is thick of
order at most 1.  Further, since $1-p(n)\ge
\frac{(1+\epsilon)\log{n}}{n}$, standard results on the connectivity of Erd{\H o}s--R\'enyi random graphs tell us that a.a.s.\ the complement of $\Gamma$ is connected, and thus that $\Gamma$ itself is a.a.s.\  not
the join of two non-trivial graphs.  Thus,
\cite{BehrstockHagenSisto:coxeter} implies that $W_\Gamma$ is not
thick of order $0$ and hence is thick of order exactly one and has
precisely quadratic divergence.

On the other hand, if \[\frac{\lambda_c-\epsilon}{\sqrt{n}}\le p(n),\]
then by Theorem~\ref{thm:no_giant} no component of the square graph can 
have full support, and thus the graph $\Gamma$ is not $\CFS$.
It then follows from \cite{Levcovitz:CFSdiv}, that $W_\Gamma$ has at least cubic
divergence, and thus by Proposition~\ref{prop:thickpolydiv} that it is not thick of order~1.
\end{proof}

\section{Branching processes and proof strategy}\label{section: strategy}

\subsection{Branching processes}\label{subsection: branching processes}
We recall here some basic facts and definitions from the theory of branching processes that we will use in our argument; for a more general treatment of such processes, see e.g.~\cite{Balister06}.
\begin{defn}
	A Galton--Watson branching process $\mathbf{W}=(W_t)_{t\in \mathbb{Z}_{\geq 0}}$ with offspring  distribution $X$ is a sequence of non-negative integer-valued random variables with $W_0=1$ and for all $t\geq 1$, $W_t=\sum_{i=1}^{W_{t-1}}X_{i,t}$, where the $X_{i,t}$: $i, t\in \mathbb{N}$ are independent, identically-distributed random variables with $X_{i,t}\sim X$ for all $i,t$.
\end{defn}  
A Galton--Watson branching process can be viewed as a random rooted tree: in the zero-th generation there is a root or ancestor, who begets a random number $X_{1,1}\sim X$ of children that form the first generation. In every subsequent generation, each child independently begets a random number of children, with the $i$-th member of generation $t$ begetting $X_{i,t}\sim X$ children.

Galton--Watson branching processes are a widely studied family of random processes and are the subject of much probabilistic research; see e.g.~\cite{Balister06} and the references therein. Here we introduce only some fairly standard elements of the theory that are needed for our argument. A Galton--Watson process $\mathbf{W}$ is said to become \emph{extinct} if $W_t=0$ for some $t\in \mathbb{N}$. The \emph{total progeny} of $\mathbf{W}$ is the total number of vertices in the associated tree, which we denote by $W=\sum_{t=0}^{\infty}W_t$; this quantity is finite if and only if $\mathbf{W}$ becomes extinct.

A key tool  in the study of $\mathbf{W}$ is the generating function of its offspring distribution, $f_X(t)=\mathbb{E}t^X$. The following standard results from the theory of branching processes relate the probability of extinction for $\mathbf{W}$ to the mean and generating function of its offspring distribution $X$. 
\begin{prop}[See e.g.~\cite{Balister06}]\label{proposition: basic branching processes}
	Let $\mu= \mathbb{E}X$ and $f(t)=f_X(t)$.  Let $\mathbf{W}$ be a
	Galton--Watson branching process with offspring distribution $X$.
	Then the following hold: 
	\begin{enumerate}[label=(\roman*)]
		\item \textbf{(subcritical regime)} if $\mu<1$, then almost surely
		$\mathbf{W}$ becomes extinct, and what is more
		\[\mathbb{P}\left(\mathbf{W} \textrm{ has not become extinct by 
			generation }k \right)=\mathbb{P}(W_k \neq 0)\leq \mu^{k}.\]
		\item \textbf{(supercritical regime)} if $\mu>1$, then the probability $\theta_e$ that $\mathbf{W}$ becomes extinct is the smallest solution $\theta \in [0,1]$ to the equation
		\[f(\theta) =\theta,\]
		and satisfies $\theta_e<1$.
	\end{enumerate}
\end{prop}
\noindent We shall also need the following result on the distribution of the total progeny $W$ of $\mathbf{W}$.
\begin{prop}[Dwass's formula~\cite{Dwass69}]\label{proposition: Dwass}
	Let $\mathbf{W}$ be a Galton--Watson branching process with offspring distribution $X$. Then the total progeny $W$ satisfies
	\[\mathbb{P}\left(W=k\right)= \frac{1}{k}\mathbb{P}\left(X_1+X_2+\cdots X_k=k-1\right),\]
	where $X_1, X_2, \ldots , X_k$ are independent, identically distributed random variables with $X_i\sim X$ for all $i\in [k]$.
\end{prop}
\subsection{Departures from the Bollob\'as--Riordan framework}

Bollob\'as and Riordan in \cite{BollobasRiordan:cliquepercolation} 
developed a powerful branching process framework for the study of clique percolation. Much of that framework can be adapted to the study of the non-induced square percolation we are concerned with in this paper. 
However there remains a number of significant hurdles which need to be overcome in order to extend their techniques to the present setting.

In the subcritical regime, the structure of squares makes the analysis of exceptional edges and offspring distributions (which are the crux of the argument) differ significantly from the Bollob\'as--Riordan paper;
care is needed to handle the resulting complications correctly. Indeed, Bollob\'as and Riordan are able to model clique percolation using a Galton--Watson branching process whose offspring distribution is roughly Poisson; however, for square percolation, the offspring distribution is more heavy-tailed, forcing us to resort to somewhat delicate technical arguments.

Further, in the supercritical regime, because of our motivation from geometric group theory, we are
interested in the study of \emph{induced} square percolation. In
particular, adding new edges to a graph
could destroy some induced squares and hence split apart
square-components even as we are trying to build a giant
square-component.  This situation is quite unlike that in clique
percolation, and we have to use a completely different sprinkling
argument to obtain our results (\emph{inter alia} sprinkling vertices
rather than edges).
Thus here again there are significant complications and major departures from Bollob\'as  and Riordan's framework in~\cite{BollobasRiordan:cliquepercolation}.

\subsection{Proof strategy}

Our results rely on the analysis of a branching
process exploration of the square-components of a graph $\Gamma\in 
\G(n,p)$ for some fixed $p=\lambda n^{-1/2}$ where $\lambda>0$.

We begin with an arbitrary induced square $S_1=abcd$ in $\Gamma$.  Its
diagonals $ac$ and $bd$ give us two pairs of non-edges which can be 
used to discover further non-edges of $\Gamma$ belonging to the same square-component.  The size of the set
$\left( \Gamma_a\cap \Gamma_c\right)\setminus \{b,d\}$ of common
neighbors of $a$ and $c$ in $V(\Gamma)\setminus\{b,d\}$ is a
binomially distributed random variable $Z \sim \mathrm{Binom}(n-4,
p^2)$.  Assuming that $\Gamma_a\cap \Gamma_c$ is an \emph{independent 
set} (i.e., contains no edge of $\Gamma$)
these common neighbors together with $b,d$ give rise to $\binom{Z
+2}{2}$ non-edges that lie in the same square-component as $ac$; 
however, since we already knew about the pair $bd$, only
$X=\binom{Z+2}{2}-1$ of these are new.  We then pursue our
exploration of the square-component of $ac$ by iterating this
procedure: for each as-yet untested non-edge $xy$ in our
square-component, we can first find the common neighbors of $xy$, and
add as ``children'' of $xy$ all the new non-edges discovered in this
way. 

This can be viewed as a Galton--Watson branching process $\mathbf{W}$ with offspring distribution $X$ in a natural way. 
Assuming the past exploration does not greatly interfere with the
distribution of the number of children in our process, the expected
number of children at each step is roughly equal to
\[\mathbb{E}X=\mathbb{E}\left(\binom{Z+2}{2}-1\right)=\mathbb{E}\left( \frac{Z^{2}+3Z}{2}\right).\]
The expected value of $X$ is readily computed from 
the first and second moments of 
the binomial distribution of $Z$, yielding
\[\mathbb{E}X = \frac{1}{2}\lambda^4+2\lambda^2+o(1).\]
The Galton--Watson process $\mathbf{W}$ becomes critical when the expectation of its offspring distribution is $1$. Solving  
\[\frac{1}{2}\lambda^4+2\lambda^2=1 \]
and selecting the non-negative root $\lambda_c= \sqrt{\sqrt{6}-2}=0.6704\ldots$, we thus see that for for any fixed $\lambda<\lambda_c$, our branching process $\mathbf{W}$ is
subcritical. We thus expect it to terminate a.a.s.\ after a fairly small
number of steps, from which one can hope to, in turn, deduce that 
a.a.s.\
all square-components are small. On the other hand, for any fixed
$\lambda> \lambda_c$, $\mathbf{W}$ is supercritical, and with
probability strictly bounded away from zero it does not terminate
before we have discovered a reasonably large number of non-edges. A 
second-moment argument can then be used to show that a strictly
positive proportion of non-edges must lie in reasonably large
square-components. With a little glueing work, we can then hope to show that in fact a strictly positive proportion of non-edges lie in a giant square-component that covers all the vertices of $\Gamma$.

The above is however a simplification of what is actually required to make the arguments go through, and the situation turns out to be considerably more nuanced than what we described above. A first issue is our assumption that the vertices in $Z$ form an independent set:  
in the subcritical regime, we need to consider what happens if the 
set $Z$ of common neighbors of some non-edge $xy$ which we are 
testing interacts with some other previously discovered vertices, or 
with vertices in $Z$. In particular, any ``exceptional'' edge from $Z$ to previously
discovered vertices other than $x$, $y$ could potentially create many
additional squares, and hence add many new pairs to our
square-component which are not accounted for by our branching
process.  Bollob\'as and Riordan faced a similar problem in their work
on clique percolation.  However, as stated in the previous subsection, the
way they dealt with ``exceptional edges'' does not quite
work for us in the square percolation setting.  One issue is that in a copy
of $C_4$, vertices on opposite sides of a diagonal are not adjacent, so that
the number of $4$-cycles created by an exceptional edge cannot always
be bounded by the degree of a newly discovered vertex.  In addition, we 
note that if it is not dealt with properly, the presence of
exceptional edges could significantly affect the future distribution
of the number of children in our branching process: if in the example
above $a,c$ had \emph{three} common neighbors among the already
discovered vertices rather than two, then the correct number of
children for $ac$ in the exploration process would be
$\binom{Z+3}{2}-3= \binom{Z+2}{2}-1 + Z$, which has expectation equal
to $1+\lambda_c^2>1$ when $\lambda=\lambda_c$.  Finally, for the
argument to work, we need not only for a Galton--Watson branching
process with offspring distribution $X$ to become a.a.s.\ extinct
within a few generations (which is an easy first moment argument): we
also need its total progeny to be a.a.s.\ small.  Here the fact that
$X$ is a quadratic function of the binomial random variable $Z$ (and
thus rather heavy-tailed) causes complicated issues, which
were not faced in \cite{BollobasRiordan:cliquepercolation} (where the
offspring distribution was essentially Poisson with mean $<1$).
Overcoming these problems is the main work done in
Section~\ref{section: lower bound}.

Secondly, in the supercritical argument, after establishing the 
a.a.s.\ existence of many non-edges in reasonably large
square-components, we must prove the a.a.s.\ existence of a giant
square-component covering all vertices and a strictly positive proportion of 
non-edges of $\Gamma$.  Here the crucial point is that, because of the applications in geometric group theory motivating our work, we are considering \emph{induced}
square percolation. The size of a largest square-component in
$\Gamma$ is not monotone with respect to the addition of edges to the
graph --- adding an edge could very well destroy an induced square,
thus potentially breaking a large square-component in to several
smaller pieces.  So we have to use a completely new sprinkling
argument to be able to agglomerate all ``reasonably large'' square-components into a single giant square-component.  To do this we reserve
some vertices for sprinkling, rather than edges.  We use these
vertices to build bridges between reasonably large square-components
in a sequence of rounds until all such components are joined into one.
Finally once  we have established the a.a.s.\ existence
of a giant square-component, some care is needed to ensure this square-component covers every vertex of
$\Gamma$. 
Assembling a giant square-component and ensuring it has full support
in this way involves overcoming a number of interesting  
obstacles, and is the main work done in Section~\ref{section: upper bound}.

\section{The subcritical regime: proof of Theorem~\ref{thm:no_giant}}\label{section: lower bound}

Theorem~\ref{thm:no_giant} will be established as an immediate 
consequence of a stronger result, Theorem~\ref{thm:strong_no_giant}, which we 
state and prove below after providing a few preliminary definitions.

Given a graph $\Gamma$, in addition to the square-graph $\square(\Gamma)$ from Definition~\ref{def: square graph} we shall consider a different but closely related auxiliary graph $\noninduced(\Gamma)$ that includes information about all squares in $\Gamma$ (rather than just the induced squares). Explicitly 
we let $\noninduced:=(V(\Gamma)^{(2)}, \{\{ac, bd\}: \ ab, bc, cd, ad
\in E(\Gamma)\})$ be the graph whose vertices are pairs of vertices from $V(\Gamma)$ and whose
edges correspond to (not necessarily induced) copies of $C_4$.  The \emph{support} $\mathrm{supp}(C)$ of a component $C$ of $\noninduced$ is defined as in Definition~\ref{def: support of a square component}, mutatis mutandis.

Note that the square graph $\square (\Gamma)$ is exactly the
subgraph of $\noninduced$  induced by the set $\{ab\in V(\Gamma)^{(2)}: \ ab\notin E(\Gamma)\}$ of non-edges of $\Gamma$. In particular, for every square-component $C$ in $\square (\Gamma)$, there is a component $C'$ in $\noninduced$ with $C\subseteq C'$ and thus
$\supp(C)\subseteq \supp(C')$.   To establish Theorem~\ref{thm:no_giant}, it is thus enough to prove the following 
stronger theorem that bounds the size of the support in $\Gamma$ of components of $\noninduced$.
\begin{thm}\label{thm:strong_no_giant}
Let $\lambda<\lambda_c$ be fixed.  Suppose that $p(n)\le \lambda
n^{-1/2}$.  Then for $\Gamma\in\G(n,p)$, a.a.s.\ every component of
$\noninduced$ has a support of size $O((\log n)^{2^{31}})$. 
\end{thm}
\noindent Since the order of the support of
the largest component in $\noninduced$ is monotone non-decreasing with
respect to the addition of edges to $\Gamma$, we may assume in the
remainder of this section that $p(n)=\lambda n^{-1/2}$.  Further,
since $\lambda<\lambda_c$ is fixed, there exists a constant
$\varepsilon>0$ such that for a binomially distributed random variable
$Z\sim \mathrm{Binom}(n, p^2)$ we have
\begin{align}\label{eq: subcritical assumption}
	\mathbb{E}\left(\binom{Z+2}{2}-1\right)=1-\varepsilon.
\end{align}
With this last equality in hand, we are now ready to present and
analyse the exploration process that lies at the heart of our proof of
Theorem~\ref{thm:strong_no_giant}.

We shall discover a superset of the component of
$\noninduced$ which contains sine fixed pair $v_1v_2\in V_\Gamma^{(2)
}$. We begin our exploration by finding common neighbors of $v_1$ and $v_2$, then adding all pairs of such newly discovered vertices to a set of unexplored pairs. These new pairs obviously lie in the same component of $\noninduced$ as $v_1v_2$.

After this initial step in the exploration, we proceed as follows.
First, we choose a new pair $a_t$ from our set of unexplored
pairs.  By assumption, there exists a  previously explored pair $b_t$
such that all $4$ edges from $a_t$ to $b_t$ are present.  We continue
our exploration by finding the set $Z_t$ of common neighbors of the
vertices in $a_t$ among the \emph{previously undiscovered} vertices of
$\Gamma$.  We then add all pairs $\left(Z_t\cup
b_t\right)^{(2)}\setminus \{b_t\}$ to our set of active pairs ---
these obviously lie in the same component of $\noninduced$ as $a_t$ --- and
delete $a_t$ from that set.  We then repeat the procedure, choosing a
new unexplored pair $a_{t+1}$, finding its common neighbors among
undiscovered vertices, etc.

This, however, is not enough to discover the totality of the component
of $v_1v_2$ in $\noninduced$.  Indeed, it is possible that the pair $a_t$ has
additional common neighbors among already discovered vertices (in
addition to the two neighbors in $b_t$), which could give rise to
additional, unexplored pairs that lie in the same component of $\noninduced$ as
$a_t$.  To deal with this possibility, we have to add in an
exceptional phase in our exploration, which takes care of
potential additional edges among the vertices we have discovered.  In this exceptional phase, we
generously overestimate how many new unexplored pairs could be
discovered, and for each of these new pairs we start new and essentially
independent versions of our exploration process, which we think of as children processes of our original exploration process.

The key is that, with the exceptional phase factored in, we do
discover a superset of the collection of all pairs in the same
component of $\noninduced$ as our starting pair $v_1v_2$.  We compare the
non-exceptional phase of our exploration to a subcritical branching
process and give upper bound on its total progeny, which is small.
Obtaining this bound is somewhat tricky (due to the nature of our
offspring distribution) and relies on a rather technical application of Dwass's
formula (Proposition~\ref{proposition: Dwass}).  The remainder of the
proof is provided in Lemma~\ref{lemma: probability of exceptional stop
is small} which shows that, with high probability, we do not run
through more than five exceptional phases.  In particular, we do not
start too many child processes, so that with high probability our
overall exploration process stops before we have discovered a large number of vertices.

\subsection{An exploration process}\label{subsection: subcritical exploration process}

Our exploration process will proceed by considering the following for 
each time $t\geq 0$:
\begin{itemize}
	\item \textbf{(Discovered vertices.)} An ordered set of 
	vertices: $D_t=\{v_1, v_2,
\ldots , v_{d_{t}}\}$. 

	\item \textbf{(Active pairs.)} A set of pairs of vertices (ordered
lexicographically with respect to the ordering on $D_t$): $A_t=\{x_1y_1, x_2y_2, \ldots ,
x_{a_t}y_{a_t}\}\subseteq D_t^{(2)}$.

	\item \textbf{(Discovered pairs.)} A set of pairs of vertices: $S_t\subseteq D_t^{(2)}$.

	\item \textbf{(Explored edge set.)} A set of edges: $E_t\subseteq 
	D_t^{(2)}\cap E(\Gamma)$.
	
	\item \textbf{(Epoch.)} An integer: $e_t \in \{0, 1, 2, 3, 4, 5\}$.
\end{itemize}

These sets will satisfy:
\begin{itemize}
	\item [($\star$)] for all $t\geq 0$ and for every active pair $x_iy_i\in A_t$, the vertices $x_i$ and $y_i$ have either $0$ or $2$ common neighbors in the graph $(D_t, E_t)$.
\end{itemize}

The initial state of the exploration consists of the following data, which is seeded 
by a choice of $v_1v_2$, an arbitrary pair of vertices from $V(\Gamma)$ (note 
that this pair can, alternatively, be thought of as a vertex of $\noninduced$).  
We set $D_0=\{v_1,v_2\}$, $A_0=S_0=\{v_1v_2\}$, $E_0=\emptyset$ and $e_0=0$. 

\medskip

At each time step $t$ our exploration proceeds as follows, with $\varepsilon$ as given in equation~\eqref{eq: subcritical 
assumption}:
\begin{enumerate}[label=\textbf{\arabic*.}]
	\item\label{exploremain} 
	If $\vert D_t\vert<2^{2^{10}}\varepsilon^{-2^{10}}(\log
	n)^{2^{31}}$ and $A_t\neq\emptyset$, let $a=\pair{x}{y}$ be the
	first pair in $A_t$.  For each $z\in V(\Gamma)\setminus D_t$ we
	test whether or not $z$ sends an edge in $\Gamma$ to both vertices
	of $a$.  Set $Z_t:=\{z\in V(\Gamma)\setminus D_t:\
	\edge{z}{x},\edge{z}{y}\in E(\Gamma)\}$ and $F_t$ to be the
	collection of joint neighbors of $x$ and $y$ in $(D_t, E_t)$. 
	(Note, by property $(\star)$, the set $F_{t}$ is  either 
	empty or consists of a pair of discovered vertices.)
	
	We arbitrarily order the vertices in $Z_t$ as $\{v_{d_t+1},
	v_{d_t+2}, \ldots, v_{d_{t+1}}\}$, and add them to $D_t$ to form
	$D_{t+1}$.  We then set \[A_{t+1}= \left(A_t\cup \left(F_t\cup
	Z_t\right)^{(2)}\right)\setminus\left(a\cup F_t^{(2)} \right)\]
	to be the new collection of active pairs (which again is ordered
	lexicographically with respect to the ordering on $D_t$), set
	$E_{t+1}=E_t\cup \{zx, zy: \ z\in Z_t\}$, set $e_{t+1}=e_t$,
	$S_{t+1}=S_t\cup A_{t+1}$ and then proceed to the next time step
	$t+1$ of the process. Note that since the only new edges being 
	added are ones connecting a new vertex to $x$ and $y$ and since 
	$xy\notin A_{t+1}$ each pair in $A_{t+1}$ has either $0$ or $2$ 
	common neighbors in $(D_{t+1}, E_{t+1})$, i.e., property ($\star$) is still satisfied in the next time-step.	

	\item\label{explorelarge} If $\vert D_t\vert \geq 2^{2^{10}}\varepsilon^{-2^{10}}(\log
n)^{2^{31}}$, then we terminate the process and declare \emph{large stop}.
\item\label{exploreexcept} If $\vert D_t\vert <  2^{2^{10}}\varepsilon^{-2^{10}}(\log n)^{2^{31}}$ and $A_t= \emptyset$, then we consider $i=\vert E(\Gamma[D_t])\setminus E_t\vert$. 

If $i=0$ or $e_t+i\geq 5$, then we terminate our exploration process and declare \emph{extinction stop} or \emph{exceptional stop}, respectively. \\Otherwise we set $e_{t+1}=e_t+i$, update $E_t$ by setting $E_{t}=E(\Gamma[D_t])$, and set $i_1=i$ (which by assumption is $>0$). We then run the following subroutines:

\begin{enumerate}[label=\textbf{3\Alph*.}]\item\label{subA} 
	Set
\[Z^1_t:=\{z\in V(\Gamma)\setminus D_t: \ z \textrm{ sends at least three edges into }D_t\}.\]
We then update our value of $i_1$, setting $i_1= \vert Z^1_{t}\vert$. 
\begin{itemize}
	\item If $i_1>0$ and $e_{t+1}+i_1>5$, then we terminate the whole exploration process and declare \emph{exceptional stop}.
	\item Else if $i_1>0$ and $e_{t+1}+i_1\leq 5$, we add $Z^1_t$ to $D_t$, update $E_t$ by setting $E_{t}=E(\Gamma[D_t])$, update $S_t$ by setting $S_t=D_t^{(2)}$.
	We then update $e_{t+1}$ to $e_{t+1}+i_1$ and run through subroutine \ref{subA} again.
	\item Otherwise $i_1=0$ and we proceed to subroutine \ref{subB}
\end{itemize}

\item\label{subB} Let
\[Z^2_t:=\{z\in V(\Gamma)\setminus D_t: \ z \textrm{ sends at least two edges into }D_t\},\]
Since subroutine \textbf{3A.} terminated with $i_1=0$ each vertex in $Z^2_t$ sends exactly two edges into $D_t$.  We set $D_{t+1}=D_t\cup Z^2_t$, and let $A_{t+1}$ consist of all pairs of vertices in $D_{t+1}$ containing at least one vertex of $Z^2_t$. Further, we set $S_{t+1}=S_t\cup A_{t+1}$.

Once this is done, we proceed to the next time-step $t+1$ in the overall exploration process, observing that property $(\star)$ has been preserved (since by construction every vertex in $Z^2_t$ has degree exactly two in $(D_{t+1}, E_{t+1})$).\end{enumerate}
\end{enumerate}

\begin{figure}[h]
\begin{minipage}[b]{.5\textwidth}
	\labellist
\small\hair 2pt
 \pinlabel {\textcolor{figgreen}{$D_t$}} [ ] at 7 8
 \pinlabel {\textcolor{figblue}{$F_t$}} [ ] at 12 40
 \pinlabel {\textcolor{figgreen}{$x$}} [ ] at 36 91
 \pinlabel {\textcolor{figgreen}{$y$}} [ ] at 36 9
 \pinlabel {\textcolor{figblue}{$f_{1}$}} [ ] at 27 51
 \pinlabel {\textcolor{figblue}{$f_{2}$}} [ ] at 46 51
 \pinlabel {\textcolor{figred}{$z_1$}} [ ] at 136 40
 \pinlabel {\textcolor{figred}{$z_2$}} [ ] at 153 40
\endlabellist
\centering
\includegraphics[scale=1.0]{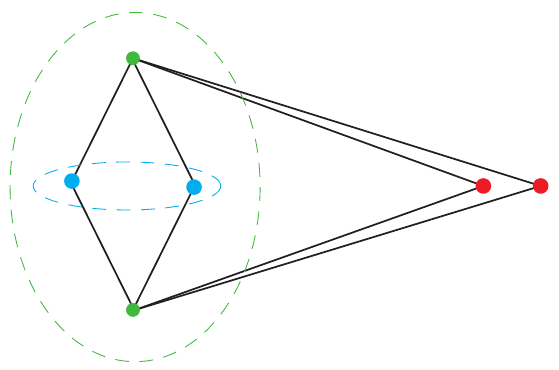}
\caption{An illustration of Stage~1 of the Exploration Process, with exploration 
from the active pair $a=xy$ at time $t$. In this example, $F_t=\{f_1,f_2\}$ and the set of newly discovered vertices is
$Z_t=\{z_1,z_2\}$. We thus have $A_{t+1}\setminus 
A_t=\{f_1 z_1, f_2 z_1, f_1 z_2, f_2 z_2, z_1z_2\}$ and $A_t\setminus A_{t+1}=\{xy\}$.
}
\label{fig:Figure_Process_1}
\end{minipage}
\hspace{-1cm}
\begin{minipage}[b]{.5\textwidth}

	\labellist
\small\hair 2pt
 \pinlabel {\textcolor{figblue}{$v_{1}$}} [ ] at 119 99
 \pinlabel {\textcolor{figblue}{$v_{2}$}} [ ] at 119 18
 \pinlabel {\textcolor{figgreen}{$w_{1}$}} [ ] at 51 64
 \pinlabel {\textcolor{figgreen}{$w_{2}$}} [ ] at 154 64
 \pinlabel {\textcolor{figred}{$z_{1}$}} [ ] at 87 65
 \pinlabel {\textcolor{figred}{$z_{2}$}} [ ] at 124 3
\endlabellist
\centering
\includegraphics[scale=1.0]{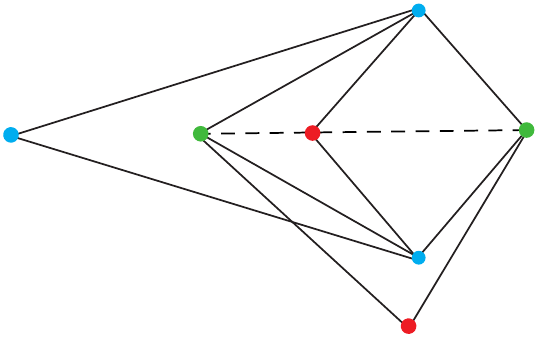}
\caption{Illustration of Stage~3 of the Exploration Process. Suppose the pairs
	$ w_{1}w_{2}, w_1z_1,w_2z_1\in S_t$ were discovered while $v_1v_2$ was active. When $w_1w_2$ is active in Stage~1 the pairs $\{v_1z_2, v_2z_2\}$ will be discovered. However the pair $z_1z_2$ will only be discovered in an instance of Stage~3, if the two dashed edges are revealed to be present (thus making $w_1z_1w_2z_2$ a square).
}
\label{fig:Figure_Process_2}
\end{minipage}
\end{figure}

\subsection{Analysing the process}\label{subsection: subcritical process analysis}
The exploration process defined in the previous subsection can terminate for one of three reasons:
\begin{enumerate}[label=(\arabic*)]
	\item $\vert D_t \vert \geq 2^{2^{10}}\varepsilon^{-2^{10}}(\log n)^{2^{31}}$  (\emph{large stop});	
	\item $e_t\geq 5$  (\emph{exceptional stop});
	\item $A_t=\emptyset$ and $E_t=\Gamma[D_t]$ (\emph{extinction stop}).
\end{enumerate}
It follows from the above that the process must in fact terminate
within $O\left((\log n)^{2^{32}}\right)$ time-steps.  We begin our
analysis by noting that, given our aim of proving
Theorem~\ref{thm:no_giant}, extinction stops are good for us:

\begin{lem}\label{lemma: good explorations give small components}
Suppose that the exploration from $\pair{v_1}{v_2}$ terminates at time
$T$ with an extinction stop.  Let $C$ be the component of $\noninduced$
containing $\pair{v_1}{v_2}$.  Then $C\subseteq S_T$.  Furthermore, the 
number of vertices in the 
support of $C$ is at most 
$2^{2^{10}}\epsilon^{-2^{10}}\left(\log{n}\right)^{2^{31}}$
and 
$\vert
C \vert \le
2^{2^{11}}\epsilon^{-2^{11}}\left(\log{n}\right)^{2^{32}}$.
\end{lem}
\begin{proof}
We perform our exploration process from the pair $v_1v_2$, and assume it terminates with an extinction stop at time $T$. It is enough to show that given $a=u_1u_2\in C\cap S_t$, for every neighbor $b=w_1w_2$ of $a$ in $\noninduced$, there is some $t'\leq T$ such that $b\in S_{t'}$.

If $a$ was discovered at a time-step where \ref{exploremain} applied  or if $a\in A_0$, then $a\in A_{t}$ and was an active pair at some time $t\geq 0$. Thus, at some later time-step $t'$ where \ref{exploremain} applies, our exploration process selects $a$ as its ``exploration pair'' and discover all neighbors $z_1z_2$ of $a$ in $\noninduced$ with $z_1z_2\in \left(Z_{t'}\cup F_{t'}\right)^{(2)}$ --- where $F_{t'}$ is the pair we used to discover $a$, and $Z_{t'}$ is the collection of joint neighbors of $u_1$ and $u_2$ that lie in  $V(\Gamma)\setminus D_{t'}$. If $b$ is in this set, then $b\in S_{t'}$.

Otherwise, if we failed to find $b=w_1w_2$ at this time $t'$,  $b$ must contain at least one vertex from $D_{t'}\setminus F_{t'}$. By property $(\star)$ of our exploration, at least one of the edges $u_iw_j$, $i,j\in \{1,2\}$ lies outside $E_{t'}$.

In particular, since we do not end with a large or exceptional stop,
this edge will be uncovered at later time step $t''$ where \ref{subA}
applies.  But by the end of \ref{subB}, all vertices sending at least
two edges into $D_{t''}$ have been added to $D_{t''}$.  Thus all
common neighbors of $a$ will be found (since $a\subseteq D_{t''}$, and
a common neighbor of $a$ has at least two neighbors in $D_{t''}$).
Hence after the updates $b\subseteq D_{t''}$.  There are two options: if both vertices
of $b$ are present after \ref{subA}, then $b\in S_{t''}$, since after
\ref{subA} all possible pairs of discovered vertices (not already
tested) are added to $S_{t''}$.  Otherwise, both vertices of $b$ are
present after \ref{subB}, and since at least one of them was
discovered in \ref{subB}, $b$ is added to $A_{t''+1}\subseteq
S_{t''+1}$, and we are done again. 

If on the other hand $a$ was discovered at a step $t$ where \ref{subB} applies, then $a$ is added to $A_{t+1}$, and the above applies. Finally, if $a$ was discovered at a time step $t$ where \ref{subA} applies, then in \ref{subA} and \ref{subB}, all common neighbors of $a$ are added to $D_t$, and all such pairs are added to $S_{t}$ or $S_{t+1}$.

Either way, since $S_t\subseteq S_{T}$ for all $t\leq T$, we see that
$b\in S_T$.  Thus every neighbor of $a$ in $\noninduced$ is eventually
discovered by our exploration process, and $C\subseteq S_T$ as
claimed.  Further, since $S_T\subseteq D_T^{(2)}$ by construction, and since 
our exploration ends with an extinction
stop by
the hypothesis of the lemma, we have $\vert C\vert \leq
\frac{1}{2}\vert D_T \vert^2 < 2^{2^{11}}\varepsilon^{-2^{11}}(\log
n)^{2^{32}}$ as claimed.
\end{proof}

We now turn to the technical crux of the analysis.

\begin{lem}\label{lemma: subcritical, domination}
If we are at a time-step $t$ of the process where \ref{exploremain} applies, then given the past history of the process, the random variable  $\vert A_{t+1}\setminus A_t\vert$ counting the number of new active pairs discovered by $a_1=x_1y_1$ is stochastically dominated by a random variable $X=\frac{Z^2+3Z}{2}$, where $Z$ is a binomial random variable with parameters $n$ and $p^2$.
\end{lem}
\begin{proof}
As observed by Bollob\'as and Riordan~\cite[Inequality (3)]{BollobasRiordan:cliquepercolation},  the past of the exploration is the intersection of the principal increasing event $\mathcal{U}_t=\{E_t\subseteq E(\Gamma)\}$ (corresponding to the edges that we have discovered are present in $\Gamma$) and a decreasing event $\mathcal{D}$ (corresponding to the intersection of a number of events of the form ``at least one of $zx, zy$ is not in $E(\Gamma)$''). In particular appealing to Harris's Lemma~\cite{Harris60}, the conditional probability that $x$ and $y$ both send an edge to $z\in V(\Gamma)\setminus D_t$ in $\Gamma$ given the history is at most the unconditional probability $p^2$. (We note here that the fact $z\notin D_t$ is essential --- we have no control over the conditional probabilities of edges inside the set of discovered vertices $D_t$.)

Thus conditional on the past history of the exploration process, the distribution of $\vert Z_t\vert$ is stochastically dominated by a random variable $Z\sim \mathrm{Binom}(n, p^2)$. The Lemma then immediately follows from the definition of our exploration process. (Note that here we are using the fact property ($\star$) is maintained throughout our exploration.)
\end{proof}
\begin{lem}\label{lemma: bound on Z sim Binom(n, p^2)}
Let $Z\sim \mathrm{Binom}(n,p^2)$ and $k\in \mathbb{N}$. Then $\mathbb{P}(Z\geq  9 \log n + 9\log k)\leq n^{-5}k^{-6}$.
\end{lem}

\begin{proof}
Recall that in this section, $p=\lambda n^{-1/2}$, for some constant $\lambda<\lambda_c$. Since $Z\sim\mathrm{Binom}(n,p^2)$, we have:
\begin{align*}\mathbb{P}\Bigl(Z\geq 9(\log n+\log k)\Bigr)&=\sum_{r=\lceil 9(\log n+\log k)\rceil }^{n} \binom{n}{r}p^{2r}(1-p^2)^{n-r}
< \sum_{r=\lceil 9(\log n+\log k)\rceil }^n n^r\left(\lambda_cn^{-1/2}\right)^{2r}\\
& = \sum_{r=\lceil 9(\log n+\log k)\rceil}^n\left(\lambda_c\right)^{2r}< n\lambda_c^{2\lceil 9(\log n+\log k)\rceil}\leq n\exp\left(18\log \lambda_c (\log n +\log k)\right), 
\end{align*}
where in the last two inequalities we used the fact $(\lambda_c)^2=\sqrt{6}-2<1$. Since 	$18\log\lambda_c< -6$, this immediately gives us the desired bound
$$\mathbb{P}\left(Z\geq 9\log n+9\log k\right)< n e^{-6\log n-6\log k}=n^{-5}k^{-6}.$$
\end{proof}	

\begin{cor}\label{corollary: bound on pairs having many common neighbors}
$\mathbb{P}\left(\exists \ \pair{x}{y}\in V(\Gamma)^{(2)} \textrm{ : }\vert \Gamma_x\cap \Gamma_y\vert \geq 9\log n\right)\leq n^{-3}$. 	
\end{cor}

\begin{proof}
Fix $\pair{x}{y}\in V(\Gamma)^{(2)}$. By Lemma~\ref{lemma: bound on Z sim Binom(n, p^2)} with $k=1$, we have
\begin{align*}
\mathbb{P}\left(\vert \Gamma_x\cap \Gamma_y\vert\geq 9\log n\right)=\sum_{r=\lceil 9\log n\rceil }^{n-2} \binom{n-2}{r}p^{2r}(1-p^2)^{n-2r}< \mathbb{P}\left(Z\geq 9 \log n\right)\leq n^{-5}.
\end{align*}
Taking a union bound over all $\binom{n}{2}<n^2$ possible choices of the pair $xy$, the lemma follows.
\end{proof}

We now analyse the total progeny of the Galton--Watson branching process with offspring distribution given by the random variable $X$ from the statement of
Lemma~\ref{lemma: subcritical, domination}. By~\eqref{eq: subcritical assumption}, $\E X=1-\epsilon$ and thus the 
branching process is subcritical.  Unfortunately, combining the Markovian
bound on the extinction time from Proposition~\ref{proposition: basic
branching processes}(i), with the bounds on the maximum
degree in $\noninduced$ from Corollary~\ref{corollary: bound on pairs having
many common neighbors} does not give us sufficiently good control on the extinction time and total progeny of our Galton--Watson process.  Thus we turn to an application of Dwass's formula to obtain the tighter bounds needed for the proof of
Theorem~\ref{thm:no_giant}.

\begin{lem}\label{lemma: total progeny}
Let $\mathbf{W}= \left(W_t\right)_{t\in \mathbb{Z}_{\geq 0}}$ be a Galton--Watson branching process with an offspring distribution $X$ as in Lemma~\ref{lemma: subcritical, domination}. Set $k_0 = 2^{26}\varepsilon^{-2}(\log n)^5$. Then $\mathbf{W}$ is subcritical, and its total progeny $W=\sum_{t=0}^{\infty}W_t$ satisfies
\[ \mathbb{P}\left(W \geq k_0 \right)=O\left(n^{-5}\right).\]
\end{lem}
\begin{proof}
Let $\{X_{k,j}: \ k \in \mathbb{N}, \ j \in [k]\}$ be an infinite family of independent, identically distributed copies of $X$. For each $k\in \mathbb{N}$ and every $j\in [k]$, write $X_{k,j}$ as $X_{k,j}= {X^a}_{k,j}+{X^b}_{k,j}$, where ${X^a}_{k,j} =\min \left(X_{k,j}, 2^8(\log n+ \log k)^2\right)$.  Set $\mu_{k}^a=\mathbb{E}({X^a}_{k,1})$. By construction, $\mu_k^a\le \E[X]=1-\varepsilon$.
Since $X=(Z^2+3Z)/2\leq 2Z^2$, where $Z\sim \mathrm{Binom}(n,p^2)$, and since $2\left(9( \log n + \log k)\right)^2< 2^8 (\log n+\log k)^2$, Lemma~\ref{lemma: bound on Z sim Binom(n, p^2)} implies that for every $k\in \mathbb{N}$ and $j\in [k]$,
\begin{align}\label{eq: bound on Pr X^b_{i,j} is nonzero}
\mathbb{P}\left({X^b}_{k,j}>0\right) \leq \mathbb{P}\left(2Z^2\geq 2^8(\log n +\log k)^2 \right)    \leq \mathbb{P}\left(Z\geq 9 \left(\log n+\log k\right)\right)\leq n^{-5}k^{-6}.
\end{align} 
Applying Dwass's formula, Proposition~\ref{proposition: Dwass},  to our branching process $\mathbf{W}$, we have that for any $k\in \mathbb{N}$,
\begin{align}\label{eq: bound on W=k part 1}
\mathbb{P}(W=k)&= \frac{1}{k}\mathbb{P}\left(\sum_{j=1}^k X_{k,j}=k-1\right) \notag\\
&\leq \frac{1}{k}\left(\mathbb{P}\left(\sum_{j=1}^k \frac{{X^a}_{k,j}-\mu_{k}^a}{2^8 (\log n+ \log k)^2}\geq \frac{k(1-\mu_{k}^a)-1}{2^8 (\log n+ \log k)^2}\right) + \mathbb{P}\left(\sum_{j=1}^k{X^b}_{k,j}>0\right)\right).
\end{align}
Since $\mu_k^a\le 1-\varepsilon$, for  $k\geq 2\varepsilon^{-1}$ we have that $2(1-\mu_k^a)-\frac{2}{k}>2(1-\mu_k^a)-\varepsilon>\varepsilon$ and hence $k(1-\mu_k^a)-1>\frac{\varepsilon k}{2}$. Thus we have
\[s:= \frac{k(1-\mu_{k}^a)-1}{2^8 (\log n+\log k)^2}>\frac{\varepsilon k}{2^9 (\log n+\log k)^2}:=s'.\]

Since the random variables $({X^a}_{k,j}-\mu_{k}^a)/(2^8(\log n+\log k)^2)$ are by construction independent random variables with mean zero and absolute value at most $1$, we can apply a standard Chernoff bound~\cite[Theorem A.1.16]{AlonSpencer:book} in~\eqref{eq: bound on W=k part 1} to obtain:

$$\mathbb{P}\left(\sum_{j=1}^k \frac{{X^a}_{k,j}-\mu_1^a}{2^8 (\log n)^2}\geq s\right)\leq e^{-\frac{(s')^2}{2k}}= e^{-\frac{\varepsilon^2 k}{2^{19}(\log n+\log k)^4}}.$$

 Letting, $k_0=k_0(n)=\lceil 2^{26}\varepsilon^{-2}(\log n)^5\rceil$ we have that for $n$ sufficiently large, all $k\geq k_0(n)$ satisfy $k\ge 5\varepsilon^{-2}\cdot2^{19}(\log{n}+\log{k})^5$. In particular for such $k$ we have

\begin{align}\label{eq: bound on the X_{i,1} summand}
\mathbb{P}\left(\sum_{j=1}^k \frac{{X^a}_{k,j}-\mu_k^a}{2^8 (\log n)^2}\geq s\right)\le e^{-\frac{\epsilon^2k}{2^{19}\left(\log n + \log k\right)^4}}\le e^{-5\left(\log n + \log k\right)}\le n^{-5}k^{-5}.\end{align}
Further, applying inequality~\eqref{eq: bound on Pr X^b_{i,j} is nonzero} and a union bound, we get that
\begin{align}\label{eq: bound on the X^b_{k,j} summand}
 \mathbb{P}\left(\sum_{j=1}^k{X^b}_{k,j}>0\right)\leq k \mathbb{P}({X^b}_{k,1}>0)\leq n^{-5}k^{-5}.
\end{align}
Combining~\eqref{eq: bound on W=k part 1}, \eqref{eq: bound on the X_{i,1} summand} and \eqref{eq: bound on the X^b_{k,j} summand}, we finally have
\begin{align*}
\mathbb{P}(W\geq k_0)&\leq \sum_{k=k_0}^{\infty}\mathbb{P}(W=k)\leq \sum_{k=k_0}^{\infty}\frac{1}{k}2n^{-5}k^{-5}= O\left( n^{-5}\right).
\end{align*}	
\end{proof}

We now use Lemma~\ref{lemma: total progeny} to estimate the probability that our exploration ends with a large stop. The key is to view our exploration as a branching process of branching processes: we begin with a subcritical branching process, corresponding to step \ref{exploremain} in our exploration process. Call this the ancestral branching process. When this process becomes extinct, we run through step \ref{exploreexcept} of our exploration process, potentially adding new active pairs to our otherwise empty set of active pairs $A_t$. For each of these new pairs, we start an independent child branching process. For each of these we repeat the same procedure as for the ancestral branching process (so the child processes can generate their own child processes, and so forth). Thus to bound the total number of pairs discovered over the entire course of our exploration, we must control the growth of this branching process of branching processes.
\begin{lem}\label{lemma: probability of a large stop is small}
	The probability that the exploration from $v_1v_2$ terminates with a large stop is $O(n^{-3})$.
\end{lem}
\begin{proof}
We view our exploration process as a kind of branching process of 
branching processes, with parent processes begetting children processes as described at the start of this section. Beginning from the single pair $v_1v_2$, the 
exploration of its component in $\noninduced$ undertaken at time-steps $t$ 
when \ref{exploremain} applies is dominated by a subcritical branching 
process $\mathbf{W}$ as in the statement of Lemma~\ref{lemma: total 
progeny}. When that process terminates, we have discovered a certain 
set $D_{t_0}$ of vertices of $\Gamma$. If \ref{exploreexcept} 
applies, then we add a certain number of vertices to $D_{t_0}$ to 
form $D_{t_0+1}$ and then add a subset of the pairs from $D_{t_0+1}^{(2)}$ to form $A_{t_0+1}$. We may view each of the pairs $a_i$ added to $A_{t_0+1}$ at this time as the root of a subcritical independent branching process $\mathbf{W}_i$. There are at most $\vert D_{t_0+1}\vert^2$ of these ``child-processes'', and they are stochastically dominated by independent copies $\mathbf{W}_i$ of the subcritical branching process $\mathbf{W}$ from Lemma~\ref{lemma: total progeny}. When all these branching processes have become extinct, we now have a (larger) set $D_{t_1}$ of discovered vertices and we may be back at a time step where \ref{exploreexcept} applies. We repeat our procedure --- adding vertices to form $D_{t'+1}$, adding new pairs to form $A_{t_1+1}$, etc.  The whole procedure can begin again at most $5$ times (for otherwise an exceptional stop must have occurred).

We run our exploration ignoring large stops and only applying \ref{exploremain} and \ref{exploreexcept} until the process terminates (with an exceptional stop or an extinction stop), and show that if an exceptional stop does not occur then with probability $O(n^{-3})$ the final size of the discovered set of vertices is at most  $2^{2^{10}}\varepsilon^{-2^{10}}(\log n)^{2^{31}}$.

Since we never can consider more than  $\binom{n}{2}$ different active pairs, it follows that we never start more than $n^2$ branching processes $\mathbf{W}_i$ in the course of our exploration. By Lemma~\ref{lemma: total progeny}, and a union bound, the probability that one of our at most $n^2$ branching processes $\mathbf{W}_i$ has a total progeny of more than $k_0=2^{26}\varepsilon^{-2}(\log n)^5$ is $O(n^{-3})$. Assume from now on this does not happen. Since the progeny of our processes correspond to pairs $xy\in V(\noninduced)=V(\Gamma)^{(2)}$, none of our processes can add more than $2k_0$ vertices to the set of discovered vertices $D_t$.

Further, by Corollary~\ref{corollary: bound on pairs having many common neighbors}, the probability that there is any pair $xy\in V(\noninduced)=V(\Gamma)^{(2)}$ such that $\vert \Gamma_x\cap \Gamma_y\vert \geq 9 \log n$ is at most $n^{-3}$. Assume from now on this does not happen. Then in the first time-step $t_0$ where \ref{exploreexcept} applies, we can bound the number of vertices added to $D_{t_0}$: each pair $xy$ can contribute at most $9\log n$ vertices to a $Z^1_t$ or $Z^2_t$, and as we do not have an exceptional stop we can repeat the addition procedure at most $5$ times. In particular we have
\begin{align*}\vert D_{t_0+1}\vert &\leq\left(\left( \left(\left( \left(\left(\vert D_{t_0}\vert\right)^2 9 \log n\right)^29\log n\right)^2 9\log n)\right)^29\log n\right)^29\log n \right)^29\log n\\
& \leq \left(\vert D_{t_0}\vert  9 \left(\log n\right) \right)^{2^6}\leq\left(18 (\log n) k_0 \right)^{2^6}\ .\end{align*}
The number of child processes started at that time step is at most $\vert D_{t_0+1}\vert^2$; by our assumption, each of these discovers at most $2k_0$ vertices in total, so that by the next time-step $t_1$ when \ref{exploreexcept} applies, we have 
\[ \vert D_{t_1}\vert \leq \vert D_{t_0+1}\vert^2 2k_0 \leq \left(18 (\log n) k_0 \right)^{2^7}2k_0:=2k_1.\]
Repeating the analysis above, we obtain 
\[\vert D_{t_1+1}\vert \leq \left(18 (\log n) k_1\right)^{2^6},\]
and in the next time-step $t_2$ where \ref{exploreexcept} applies we have
\[ \vert D_{t_2}\vert \leq   \vert D_{t_1+1}\vert^2 2k_0 \leq \left(18 (\log n) k_1 \right)^{2^7}2k_0   =:2k_2.\]
and we can keep going in this way, defining $k_3$, $k_4$ mutatis mutandis. If we avoid an exceptional stop, then we must terminate by the fifth time-step $t_4$ when \ref{exploreexcept} applies. Iterating  our analysis, we see that the size of the final set of discovered vertices $D_{t_4}$ is as most
\begin{align*}
\vert D_{t_4}\vert &\leq \left(18 (\log n) k_3 \right)^{2^7}2k_0\\
& = \left(9 (\log n) \left(9 (\log n) \left(9 (\log n) \left(18 (\log n)k_0\right)^{2^7}2k_0 \right)^{2^7}2k_0 \right)^{2^7}2k_0 \right)^{2^7}2k_0\\
&< \left(18(\log n)k_0 \right)^{2^{29}}<2^{2^{10}}\varepsilon^{-2^{10}}(\log n)^{2^{31}}.
\end{align*}
This shows that the probability our  process terminates with a large stop is $O(n^{-3})$.		
\end{proof}

Finally, we compute the probability that an exploration ends with an
exceptional stop. 

\begin{lem}\label{lemma: probability of exceptional stop is small}
The probability that the exploration from $v_1v_2$ terminates with an exceptional stop is $O((\log n)^{30\cdot2^{31}}n^{-5/2})$.	
\end{lem}
\begin{proof}
Suppose that the exploration from $v_1v_2$ terminates at time $T$ with an exceptional stop.
Then we must have discovered at least $5$ exceptional edges at time-steps $t\leq T$ when \ref{exploreexcept} applied, where we call an edge exceptional if it appeared in $E(\Gamma[D_t])\setminus E_t$ (type 1) or as the third or above edge  from some vertex $z\in Z_t^1$ to $D_t$ (type 2), and where such edges are ordered according to the ordering of the vertices of $D_t$.

Since the exploration did not terminate with a large stop, at each of these time steps we had $\vert D_t\vert \leq 2^{2^{10}}\varepsilon^{-2^{10}}(\log n)^{2^{31}}:= \Delta$ at the start of each time-step $t$. Also, since we did not terminate with a large stop, the time $T$ at which the process terminated must satisfy $T\leq \Delta^2$ (this is an upper bound on the number of pairs we could have tested at time steps where \ref{exploremain} or \textbf{3}. applied).

In any time-step $t$ where \textbf{3}.  applies and we are testing for membership in one of the (at most five) sets $Z^1_t$ considered in that turn, the probability that a vertex in $V(\Gamma)\setminus D_t$ sends at least three edges of $\Gamma$ to the set $D_t$ is at most 
\[\sum_{i=3}^{\vert D_t\vert} \binom{\vert D_t\vert }{i}p^i(1-p)^{\vert D_t\vert -i} \leq  \sum_{i=3}^{\Delta}\Delta^i p^i= O\left(\Delta^4 p^3\right).\]
Since we have at most $n$ vertices in $V(\Gamma)\setminus D_t$ and  at most $T\leq \Delta^2$ time-steps to choose from, the probability of having found at least $j$ type 2 exceptional edges for some $1\leq j \leq 5$ is 
\begin{align}\label{eq: bound on type 1 exceptional edges}
O\Bigl(T^j \left(n \Delta^{4}p^{3}\right)^j\Bigr)= O\left(\Delta^{6j} n^{-j/2}\right).
\end{align}
A similar (but simpler) calculation yields that the probability of having found $5-j$ edges of type 1 is:
\begin{align}\label{eq: bound on type 2 exceptional edges}
O\Bigl(T^{5-j} {\left(\Delta^2p\right)}^{5-j}\Bigr)= O\left(\Delta^{4(5-j)}n^{-(5-j)/2}\right).
\end{align}
Adding the bounds \eqref{eq: bound on type 1 exceptional edges} and \eqref{eq: bound on type 2 exceptional edges} together and substituting in the value of $\Delta$, the Lemma follows.
\end{proof}	

We are now in a position to prove 
Theorem~\ref{thm:strong_no_giant}, which, as previously noted,  
immediately implies Theorem~\ref{thm:no_giant}.

\begin{proof}[Proof of Theorem~\ref{thm:strong_no_giant}] 
	Let $v_1v_2$ be an arbitrary pair of vertices from $V(\Gamma)$.
	By Lemmas~\ref{lemma: probability of a large stop is small} 
	and~\ref{lemma: probability of exceptional stop is small}, with
	probability $1-O((\log n)^{30\cdot2^{31}} n^{-5/2})=1-o(n^{-2})$
	the exploration from $v_1v_2$ terminates with an extinction stop. 
	By Lemma~\ref{lemma: good explorations give small components} we 
	obtain a bound on the size of each 
	component of $v_1v_2$ in $\noninduced$ found by this 
	exploration, and a bound on its support as well. 
	By a simple union bound, with
	probability $1-o(1)$, all pairs $v_1v_2$ lie in components of
	$\noninduced$ supported on sets of size at most
	$2^{2^{10}}\varepsilon^{-2^{10}}(\log n)^{2^{31}}$ in 
	$V(\Gamma)$.	
\end{proof}

\section{The supercritical regime: proof of Theorem~\ref{thm:1giant}}\label{section: upper bound}
Fix $\lambda>\lambda_c$. 
Suppose $\lambda n^{-1/2}\leq p(n)\leq (1-f(n))$, where $f(n)$ is a 
function with $f(n)=o(1)$ and $f(n)=\omega(n^{-2})$.  Let $\Gamma\in \G(n,p)$.  
By \cite[Theorem~5.1]{BehrstockFalgasRavryHagenSusse:StructureRandom},
we know that if $p(n)\geq 5\sqrt{\log{n}/n}$, 
then a.a.s.\ there is a square-component covering all of $V(\Gamma)$.  
We may thus restrict our attention in the proof of
Theorem~\ref{thm:1giant} to the range $\lambda n^{-1/2}\leq p(n)\leq
5\sqrt{\log{n}/n}$.  For such $p(n)$, there exists $\varepsilon>0$
such that if $Z\sim\mathrm{Binom}(n,p^2)$, then
\begin{align}\label{eq: supercritical expectation bound}
\mathbb{E} \left(\frac{Z^2+3Z}{2}\right)\geq 1+\varepsilon.
\end{align}
We shall prove the existence of a giant square-component with full 
support  in four stages: first, we define an exploration process in 
$\square(\Gamma)$. In a second stage, we analyse the process to show 
that a.a.s.\ a large proportion of non-edges of $\Gamma$ lie in 
``somewhat large'' square-components of $\square(\Gamma)$. Next, in the (more involved) 
third stage of the argument, we perform \emph{vertex-sprinkling} to show a.a.s.\ a large proportion of non-edges of $\Gamma$ lie in a giant square-component. Finally, we show there a.a.s. a giant square-component covering all of $V(\Gamma)$.

\subsection{An exploration process}\label{subsection: supercritical exploration process}

We consider an exploration process consisting of the following data 
at each time $t\geq 0$: 

\begin{itemize}
	\item \textbf{(Discovered vertices.)} An ordered set of 
	vertices: $D_t=\{v_1, v_2, \ldots , v_{d_{t}}\}\subseteq V(\Gamma)$.

	\item \textbf{(Active pairs.)} A set of pairs of vertices (ordered
lexicographically with respect to the ordering on $D_t$): $A_t=\{x_1y_1, x_2y_2, \ldots ,
x_{a_t}y_{a_t}\}\subseteq D_t^{(2)}\setminus E(\Gamma)$. 

	\item \textbf{(Reached pairs.)} A set of pairs of vertices: 
	$R_t\subseteq D_t^{(2)}\setminus E(\Gamma)$.
\end{itemize}

These sets will satisfy:
\begin{itemize}
	\item [($\star$)] for every active pair $x_iy_i\in A_t$, 
	the vertices $x_i$ and $y_i$ have at least $2$ common neighbors 
	in the subgraph of $\Gamma$ induced by $D_t$.
\end{itemize}

The initial state $t=0$ of the exploration consists of an 
arbitrary induced $C_4$ of $\Gamma$, denoted $v_1v_2v_3v_4$, and the sets:
$D_0=\{v_1, v_2,v_3,v_4\}$, $A_0=\{v_1v_3, v_2v_4\}$, and $R_0=\emptyset$.

Our exploration then proceeds as follows: 
\newline \textbf{1.  If
$\vert R_t\vert +\vert A_t\vert > (\log n)^4$}, then we terminate the
process.  
\newline \textbf{2.  If $\vert R_t\vert +\vert A_t\vert \leq
(\log n)^4$ and $A_t\neq \emptyset$}, then for each $z\in
V(\Gamma)\setminus D_t$ we test whether or not $z$ sends an edge in
$\Gamma$ to both of $\{x_1,y_1\}$ which are the vertices of the first 
pair $a_1=x_1y_1$ in the ordered set $A_{t}$. We then set $Z_t:=\{z\in
V(\Gamma)\setminus D_t:\ zx_1,zy_1\in E(\Gamma_t)\}$.  Denote by $F_t$
the set of common neighbors of $x_1$ and $y_1$ in $D_t$ (which by
property $(\star)$ has size at least $2$).

We then set $A_{t+1}=\left(A_t\setminus\{x_1y_1\} \right) \cup \left(\left(F_t\cup Z_t\right)^{(2)}\setminus \left(F_t^{(2)}\cup E(\Gamma)\right)\right)$, $R_{t+1}=R_t\cup\{x_1y_1\}$ and $D_{t+1}=D_t\cup Z_t$. We then proceed to the next time-step in the exploration process, noting that property $(\star)$ is maintained.
\newline \textbf{3. If $\vert R_t\vert \leq (\log n)^4$ and $A_t=\emptyset$}, then we terminate the process.

	\subsection{Many non-edges in somewhat large components}\label{section: set up}
The exploration process defined in the previous subsection can terminate for one of two reasons:
\begin{enumerate}[label=(\arabic*)]
	\item ({\bf Large stop}.) $\vert R_t\vert+\vert A_t\vert  > (\log n)^4$, or
	\item ({\bf Extinction stop}.) $A_t=\emptyset$.
\end{enumerate}
The process always terminates after some number $T\leq \left(\log n\right)^{4}$ of time-steps. By construction, at all times $t\geq 0$ the collection of pairs $A_t\cup R_t$ is a subset of the square-component of $v_1v_3$ and $v_2v_4$ in $\square(\Gamma)$. Our aim is to show that with somewhat large probability  $\vert A_T\vert +\vert R_T\vert >(\log n)^4$.
\begin{lem}\label{lemma: supercritical offspring distribution}
At any time-step $t\geq 0$, the distribution conditional on the past history of the process of the random variable  $X_t=\vert A_{t+1}\setminus A_t\vert$ counting the number of new active pairs discovered by $a_1=x_1y_1$  stochastically dominates a random variable $X'$ with mean $\mathbb{E}(X')\geq 1+\varepsilon+o(1)$.
\end{lem}	
\begin{proof}
Write $E_t$ for the set of edges discovered by the process.  As
observed by Bollob\'as and Riordan~\cite[Inequality
(3)]{BollobasRiordan:cliquepercolation}, the past of the process is
the intersection of the principal increasing event
$\mathcal{U}=\{E_t\subseteq E(\Gamma)\}$ and a decreasing event
$\mathcal{D}$ (corresponding to the intersection of a number of events
of the form ``at least one of $zx, zy$ is not in $\Gamma$'' for some
previously tested pair $xy$).  In particular, for $xy\in A_t$ and
$z\in V(\Gamma)\setminus D_t$,
\[\mathbb{P}\left(xz, yz\in E(\Gamma)\vert\mathcal{D}\cap\mathcal{U}\right) = \mathbb{P}\left(xz, yz\in E(\Gamma)\vert\mathcal{D}\right).\]
Let $\mathcal{D}'(z)$ denote the decreasing event
\[\mathcal{D}'(z)=\bigcap_{x'y'\in D_t^{(2)}\setminus\{xy\}} \{\textrm{at least one of  $zx', zy'$  is not in $\Gamma$}\}.\]
Note that  the event $\{xz, yz\in E(\Gamma)\}$ is independent of $\mathcal{D}\setminus \mathcal{D}'(z)$.
Using this fact and appealing to Harris's Lemma~\cite{Harris60}, we have
\[\mathbb{P}\left(xz, yz\in E(\Gamma)\vert\mathcal{D}\right) =\mathbb{P}\left(xz, yz\in E(\Gamma)\vert\mathcal{D}\cap \mathcal{D}'(z)\right)
\geq \mathbb{P}\left(xz, yz\in E(\Gamma)\vert\mathcal{D}'(z)\right).\]
Now conditional on $\mathcal{D}'$, the probability that both $xz$ and $yz$ are in $E(\Gamma)$ is readily computed: it is equal to 
\begin{align*}
\frac{p^2(1-p)^{\vert D_t\vert-2}}{p^2(1-p)^{\vert D_t\vert-2}+ 2p(1-p)(1-p)^{\vert D_t\vert-2} +(1-p)^2\left((1-p)^{\vert D_t\vert -2}+(\vert D_t\vert -2)p(1-p)^{\vert D_t\vert -3}\right) }
\end{align*}
which is equal to $p^2\left(1+O(\vert D_t\vert p)\right)=p^2+o(p^2)$ (since$\vert D_t\vert \leq 2\left(\vert R_t\vert+\vert A_t\vert \right)\leq 2(\log n)^4$).

The indicator function of the event $Y_z=\{z\in Z_t\}$ thus stochastically dominates a Bernoulli random variable with mean $p^2+o(p^2)$. Further the $(Y_z)_{z\in V(\Gamma)\setminus D_t}$ are independent events given our conditioning (since each such event is only affected by the state of edges from $D_t$ to $z$). The random variable $\vert Z_t\vert $ thus stochastically dominates a random variable $Z'\sim\mathrm{Binom}(n-\vert D_t\vert, p^2+o(p^2))$. Let $X'$ denote the sum of $\frac{(Z')^2+3Z'}{2}$ independent Bernoulli random variables with parameter $1-p$. For $\vert D_t\vert \leq 2(\log n)^4$, we have
 \begin{align*}
\mathbb{E}  \vert A_{t+1}\setminus A_t\vert& =\mathbb{E} \left(\frac{\vert Z_t\vert^2+3\vert Z_t\vert}{2}-\vert E(\Gamma)\cap (F_t\cup Z_t)^{(2) }\vert\right)\\
&\geq \mathbb{E}X'=(1-p)\mathbb{E}\left(\frac{\vert Z'\vert^2+3\vert Z'\vert}{2}\right) \geq 1+\varepsilon +o(1),
 \end{align*}
where the inequality follows from the stochastic domination of $Z'$ by
$Z_t$, and where in the last line we have used~\eqref{eq:
supercritical expectation bound} and the fact that a binomial distribution
with parameters $n-\vert D_t\vert$ and $p^2+o(p^2)$ is close to
$\mathrm{Binom}(n,p^2)$.
\end{proof}
Let $\theta_e=\theta_e(n,p)$ denote the extinction probability of the supercritical branching process $\mathbf{W}$ with the offspring distribution $X'$ given in the proof of Lemma~\ref{lemma: supercritical offspring distribution}. Note that by Proposition~\ref{proposition: basic branching processes}(b), $\theta_e$ is bounded away from $1$.

Up to the time when it terminates, our exploration  process on the 
square-component of $v_1v_2v_3v_4$ stochastically dominates 
$\mathbf{W}$. We now use this fact to show many non-edges of $\Gamma$ lie 
in ``somewhat large'' square-components.

\begin{lem}[Many squares in large square-components]\label{lemma: many squares in large components}
Fix $\lambda>\lambda_c$, $p(n)$ satisfying $\lambda 
n^{-1/2}\leq p(n)\leq  5n^{-1/2}\sqrt{\log n} $, and 
$\theta_e=\theta_e(n,p)$ as above. Then, with probability $1-o(n^{-2})$ the number $N$ of induced $C_4$s in $\Gamma$ which are part of square-components of order at least $(\log n)^4$ satisfies 
\[N= (1+o(1))\mathbb{E}N\geq 3p^4(1-p)^2\binom{n}{4}(1-\theta_e)(1+o(1)).\]
\end{lem}	
\begin{proof}
Given a collection of $4$ vertices $S\in V(\Gamma)^{(4)}$, let $E_S$ be the indicator function of the event that $\Gamma[S]\cong C_4$ and that our exploration  process from $S$ terminates with a large stop (which is equivalent to $S$ being part of a square-component of order at least $(\log n)^4$). Conditional on  $\Gamma[S]\cong C_4$, Lemma~\ref{lemma: supercritical offspring distribution} implies that $\mathbb{P}(E_S=1)\geq 1-\theta_e$ (which is the probability that the  branching process $\mathbf{W}$ does not become extinct). Applying Wald's identity, the expectation $\mu_N$ of $N$ thus satisfies
\begin{align}
\mu_N=\mathbb{E}N&=\mathbb{E}\sum_{S\in V(\Gamma)^{(4)}} E_S=\sum_{S\in V(\Gamma)^{(4)}}  \mathbb{P}(\Gamma[S]\cong C_4)\mathbb{P}\left(E_S=1\vert \Gamma[S]\cong C_4 \right)\notag\\
&\geq 3p^4(1-p)^2\binom{n}{4}(1-\theta_e). \label{eq: lower bound on mu_N}
\end{align}
We now use Chebyshev's inequality to show $N$ is concentrated around 
its mean. To do this, we must bound $\mathbb{E}N^2=\sum_{S,S'\in 
V(\Gamma)^{(4)}}\mathbb{E}E_SE_{S'}$. Consider two collections of $4$ 
vertices $S,S'\in V(\Gamma)^{(4)}$. 
\begin{claim}\label{claim: variance argument, bound on contribution from disjoint S, S'}
If	 $S\cap S'=\emptyset$,  then $E_S$ and $E_{S'}$ satisfy
\[\mathbb{E} \left(E_SE_{S'}\right)=\mathbb{E}\left(E_S\right)\mathbb{E}\left(E_{S'}\right)+ O\left((\log n)^9n^{-1}\mathbb{E}(E_S)\right).\]
\end{claim}	
	\begin{proof}
Our claim is that $E_S$ and $E_{S'}$ are essentially independent. Indeed, let us first perform our exploration process from $S$ (stopping immediately if $\Gamma[S]$ does not induce a copy of $C_4$). For $Z\sim \mathrm{Binom}(n, p^2)$ (and $\lambda n^{-1/2}\leq p\leq 5n^{-1/2}(\log)^{-1/2}$ as everywhere in this section), we have 
	\begin{align}\label{eq: useful bound on Z geq 2^7 log n}
	\mathbb{P}(Z\geq  2^7\log n)=\sum_{r=\lceil 2^7\log n\rceil}^{n} \binom{n}{r}p^{2r}(1-p^2)^{n-r}\leq \sum_{r\geq 2^7\log n} \left(\frac{en}{r}\cdot25\frac{\log n}{n}\right)^r=o(n^{-5}).\end{align}
Thus with probability $1-o(n^{-3})$,  the number of vertices added to $D_t$ in the last stage of the exploration process from $S$ is at most $2^7 \log n$, implying that the set $D_{S}$ of vertices discovered by the process from $S$ has size at most $2(\log n)^4 + 2^7 \log n$. Further, the exploration process from $S$ tests at most $\left(\log n\right)^4$ pairs in total. This allows us to bound the probability that the exploration process from $S'$ interacts with the exploration process from $S$.

First of all, by Markov's inequality the probability that a vertex in $S'$ is discovered by the process from $S$ is at most 
\[4p^2\left(\log n\right)^4 = O\left((\log n)^5 n^{-1}\right).\] Secondly, the exploration process from $S'$ tests at most $(\log n)^4$ pairs $xy$. By Markov's inequality again, the probability that some $z\in D_S$ sends an edge to both vertices in such a pair is at most
\[(\log n)^4\vert D_S\vert p^2=O \left((\log )^9 n^{-1}\right).\]
In particular, the probability of $E_{S'}=1$ given $E_S=1$ differs from $\mathbb{E} E_{S'}$ by at most $O((\log n)^{9}n^{-1}$, as claimed.
\end{proof}
\begin{claim}\label{claim: variance argument, bound on contribution from non-disjoint S, S'}
For any $S\in V(\Gamma)^{(4)}$, we have	
	\begin{align*}
	\sum_{S'\in V(\Gamma)^{(4)}:\ S\cap S'\neq \emptyset}\mathbb{E}\left(E_SE_{S'}\right)=O\left(n^3p^4\mathbb{E}(E_S)\right).&& (\dagger)\end{align*}
\end{claim}
\begin{proof}
	Fix $S$ and consider the various ways in which $S$ and $S'$ could 
	intersect non-trivially. \begin{itemize}
		\item There is one choice of $S'$  with $S=S'$, for which we have $\mathbb{E}(E_SE_{S'})=\mathbb{E}(S)$.
		\item Next, we have at most $4n$ choices of $S'$ with $\vert S\cap S'\vert =3$. Write $S=\{a,b,c,d\}$ and $S'=\{a,b,c,d'\}$. For $E_SE_{S'}$ to be non-zero, both $S$ and $S'$ must induce copies of $C_4$ in $\Gamma$ and moreover $E_S$  must occur. This is only possible if $E_S=1$ and $d'$ sends edges to the two neighbors of $d$ in $\{a,b,c\}$. Arguing as in Claim~\ref{claim: variance argument, bound on contribution from disjoint S, S'}, these two events are almost independent and occur with probability $(1+o(1))p^2\mathbb{E}E_S$. Thus the contribution of $S'$ with $\vert S\cap S'\vert =3$ to the left-hand side of $(\dagger)$ is at most $O(np^2\mathbb{E}(E_S))$.
		\item There are at most $4n^3$ choices of $S'$ with $\vert S\cap S'\vert =1$. For $E_SE_{S'}$ to be non-zero, it is necessary for $S'$ to induce a copy of $C_4$ in $\Gamma$ and for $E_S=1$. Arguing as in Claim~\ref{claim: variance argument, bound on contribution from disjoint S, S'}, these two events are almost independent and occur with probability $(1+o(1))p^4\mathbb{E}E_S$. Thus the contribution of $S'$ with $\vert S\cap S'\vert =1$ to the left-hand side of $(\dagger)$ is at most $O(n^3p^4\mathbb{E}(E_S))$.
		\item Finally, there are at most $6n^2$ choices of $S$ with $\vert S\cap S'\vert =2$. For $E_SE_{S'}$ to be non-zero, it is necessary for the vertices in $S'\setminus S$ to be incident at least three edges in $\Gamma[S']$ and for $E_S=1$. Arguing as in Claim~\ref{claim: variance argument, bound on contribution from disjoint S, S'}, these two events are almost independent and occur with probability at most $(1+o(1))p^3\mathbb{E}E_S$. Thus the contribution of $S'$ with $\vert S\cap S'\vert =2$ to the left-hand side of $(\dagger)$ is at most $O(n^2p^3\mathbb{E}(E_S))$.
	\end{itemize}
Since $np=\omega(1)$, we have $O(1+ np^2+n^2p^3+n^3p^4)=O(n^3p^4)$, and the analysis above shows the left-hand side of $(\dagger)$ is at most $O(n^2p^3\mathbb{E}(E_S))$, as claimed.	
\end{proof}
Together, Claims~\ref{claim: variance argument, bound on contribution from disjoint S, S'} and~\ref{claim: variance argument, bound on contribution from non-disjoint S, S'} imply $\mathbb{E}N^2\leq \left(\mathbb{E}N\right)^2 + O\left(n^3p^4 \mathbb{E}N\right)$. By inequality~\ref{eq: lower bound on mu_N}, we know $\mu_N=\mathbb{E}N=\Omega(n^4p^4)$. Since $p=\Omega(n^{-1/2})$, it follows that 
\[\mathrm{Var}(N)= O\left(n^2p^3 \mathbb{E}(N)\right)= O\left(\frac{\mu_N}{n}\right)=O\left(\frac{(\mu_N)^2}{n^3}\right).\]
Applying Chebyshev's inequality yields that with probability at least $1-o(n^{-2})$,
\[N=(1-o(1))\mu_N \geq 3p^4(1-p)^2\binom{n}{4}(1-\theta_e)(1-o(1)),\]
as desired.
\end{proof}

\begin{cor}\label{corollary: lots of non-edges in large-ish components (1-eps_1) version}
Let $\lambda>\lambda_c$ be fixed, and let $p=p(n)$ be an edge-probability satisfying $\lambda n^{-1/2}\leq p\leq 5 n^{-1/2}\sqrt{\log n}$. Then for all $\varepsilon_1> 0$ sufficiently small, there exist a constant $\varepsilon_2>0$  
such that if $\Gamma_1\in \G((1-\varepsilon_1)n, p(n))$, then with probability $1-O(n^{-1})$ the number $N_v$ of non-edges of $\Gamma_1$ that lie in square-components of $\square(\Gamma_1)$ of order at least $(\log n)^4$, satisfies 
\[N_v=(1+o(1))\mathbb{E}(N_v)\geq \varepsilon_2 n^2.\]
\end{cor}

\begin{proof}
Let $\lambda'=	(\lambda+\lambda_c)/2$. For $\varepsilon_1$ 
sufficiently small, we have $p(n)\geq 
\lambda'(n(1-\varepsilon_1))^{-1/2}$. We now consider $\Gamma_1\in \G(n(1-\varepsilon_1), p)$.

Ideally, we would now like to directly apply Lemma~\ref{lemma: many
squares in large components} in $\Gamma_1$.  However, 
to ensure the stochastic domination in
Lemma~\ref{lemma: supercritical offspring distribution}, we started
our exploration process from an induced $C_4$ rather than a non-edge
--- so we know that $\Omega(n^4p^4)$ induced $C_4$s are part of
square-components of order at least $(\log n)^4(1+o(1))$ whereas we
want to show $\Omega(n^2)$ non-edges lie in such components.  Since
some non-edges could have as many as $\Omega(\log n)$ common neighbors
in $\Gamma_1$, it would in principle be possible for $p$ of order
$n^{-1/2}$ that, for example, the collection of the diagonals of the induced
$C_4$s contained in such ``large'' components consists of a set of only
$O(n^2/(\log n)^2)$ non-edges.  We must thus rule out situation.

The simplest way to do this is to run through our proof of Lemma~\ref{lemma: many squares in large components} again, but this time for the variant of our exploration process from Section~\ref{subsection: supercritical exploration process} where we begin with an arbitrary non-edge $v_1v_2$ of $\Gamma_1$, set $D_0=\{v_1,v_2\}$, $A_0=\{v_1v_2\}$ and $R_0=\emptyset$. We say such an exploration \emph{survives infancy} if at the first time-step the pair $v_1v_2$ discovers a set $Z_1$ of joint neighbors that spans at least one non-edge $v_3v_4$ of $\Gamma$. 

For $p$ in the range we are considering  the random graph $\Gamma$ a.a.s.\ does not contain a complete graph on $6$ vertices, and we can use this to give a constant order lower bound on $\theta _S$, the probability the process survives infancy:
\begin{align*}
\theta_S&\ge \Pb\left(\vert \Gamma_{v_1}\cap\Gamma_{v_2}\vert\geq 6\ | v_1v_2\notin E(\Gamma)\right)-\mathbb{P}\left(\Gamma \textrm{ contain a clique on $6$ vertices}\right)\\
&\geq \binom{n-2}{6}p^{12}(1-p^2)^{n-8} -\binom{n}{6}p^{\binom{6}{2}} = \Theta((np^2)^6) = \Theta(1).
\end{align*}

Conditional on surviving infancy, by Lemma~\ref{lemma: supercritical offspring distribution} the exploration process from $v_3v_4$ stochastically dominates a supercritical branching process $\mathbf{W}$ with extinction probability $\theta_e=\theta_e((1-\varepsilon_1)n, p)$. Applying Wald's identity, this implies that the number  $N_v$ of non-edges of $\Gamma_1$ that belong to square-components of order at least $(\log n)^4$ satisfies
\[\mathbb{E}(N_v)\geq \binom{(1-\varepsilon_1)n}{2}(1-p)\theta_S (1-\theta_e)=\Omega(n^2).\]
We now bound $\mathbb{E}(N_V)^2$ much as we did in Lemma~\ref{lemma: many squares in large components}. Given a pair $xy\in V(\Gamma)^{(2)}$, write $E_{xy}$ for the event that $xy$ is a non-edge and that our exploration process from $xy$ terminates with a large stop. Claim~\ref{claim: variance argument, bound on contribution from disjoint S, S'} from the proof of Lemma~\ref{lemma: many squares in large components} shows mutatis mutandis that if $\{x,y\}\cap \{x',y'\}=\emptyset$ then
\[ \mathbb{E}\left(E_{xy}E_{x'y'}\right)=\mathbb{E}(E_{xy})\mathbb{E}(E_{x'y'}) +O\left((\log n)^{9}n^{-1}\right).\]
For non-disjoint pairs $\{x,y\}$ and $\{x',y'\}$, the situation is actually easier than it was in Claim~\ref{claim: variance argument, bound on contribution from non-disjoint S, S'}: such pairs contribute at most $2n\mathbb{E}N_v$ to $\mathbb{E}\left((N_v)^2\right)$. Thus 
\[\mathrm{Var}(N_v)= O\left(n \mathbb{E}(N_v)\right)= O\left( n^{-1} (\mathbb{E}(N_v))^2\right),\] and we conclude that with probability $1-O(n^{-1})$ there are  $(1+o(1))\mathbb{E}(N_v)=\Omega(n^2)$ non-edges contained in square-components of order at least $(\log n)^4$. The Corollary then follows from a suitable choice of 
the constant $\varepsilon_2$.
\end{proof}

\subsection{A connecting lemma}\label{section: connecting lemma}
The key to our sprinkling argument is the following, which we use to
connect the somewhat large square-components into even larger
square-components. We connect square-components by sprinkling in vertices, and looking for complete bipartite graphs with bipartition $\{x_1, x_2, y_1, y_2\}\sqcup\{z_1,z_2\}$, where $x_1x_2$, $y_1y_2$ are non-edges in distinct square-components, and $z_1z_2$ is a non-edge inside the set of newly sprinkled vertices --- see Figure~\ref{fig:connecting} below.

Recall that a $p$-random bipartite graph with partition $V\sqcup W$ is a graph on the vertex set $V\sqcup W$ obtained by including each pair $\{v,w\}$ with $v\in V, w\in W$ as an edge independently at random with probability $p$. 
\begin{lem}[Connecting Lemma]\label{lemma: connecting}
	Let $\lambda>\lambda_c$, $\delta \in (0, \frac{1}{2})$ and $\varepsilon_1, \varepsilon_2>0$ be fixed. Let $V$ be a set of  $(1-\delta)n$ vertices, and $W$ be a set of $\frac{\varepsilon_1n}{2\log_2 n}$  vertices disjoint from $V$. Suppose we are given disjoint subsets $C_1, C_2, \ldots C_r$ of $V^{(2)}$ and a subset $S \subseteq W^{(2)}$ with the following properties:
	\begin{enumerate}
		\item $\vert S\vert \geq \frac{(\varepsilon_1)^2n^2}{8(\log_2 n)^2} $;
		\item $\vert C_i\vert \geq M$ for every $i$: $1\leq i \leq 
		r$, and some $M$ satisfying: $(\log n)^4\leq M\leq \frac{\varepsilon_2}{4}n^2$;
		\item $\sum_i \vert C_i \vert \geq \varepsilon_2 n^2$.
	\end{enumerate} 	
	Let $p=p(n)$ be an edge probability with
	\[	\lambda \frac{1}{\sqrt{n}} < p(n) < 5\frac{\sqrt{\log n}}{\sqrt{n}}.\]
	Consider the $p$-random bipartite graph $B_p(V, W)$ with bipartition $V\sqcup W$. Let \textrm{Boost} be the event that for every $C_i$ with $\vert C_i \vert \leq 2M$ there exists $C_j\neq C_i$ and a triple $(x_1x_2, y_1y_2, z_1z_2) \in C_i \times C_j \times S$ such that the restriction $B_p(\{x_1,x_2,y_1,y_2\}, \{z_1,z_2\} )$ of $B_p(V,W)$ to $\{x_1,x_2,y_1,y_2\}\sqcup\{z_1,z_2\}$ is complete. Then for all $n$ sufficiently large we have
	\[\mathbb{P}(\textrm{Boost})\geq 1- \exp \left( -\frac{\varepsilon_2(\varepsilon_1)^2}{2^{16}}(\log n)^2 \right). \]
	
\end{lem}
The proof of the connecting lemma relies on a celebrated inequality of Janson and some careful book-keeping.	

\begin{prop}[The extended Janson inequality~\cite{JansonLuczakRucinski11}]\label{proposition: Janson inequality}
	Let $U$ be a finite set and $U_q$ a $q$-random subset of $U$ for some $q\in [0,1]$. Let $\mathcal{F}$ be a family of subsets of $U$, and for every $F\in \mathcal{F}$ let $I_F$ be the indicator function of the event $\{F\subseteq U_q\}$. Set $I_{\mathcal{F}}=\sum_{F\in \mathcal{F}} I_F$, and let $\mu=\mathbb{E} I_{\mathcal{F}}$ and $\Delta=\sum_{F,F' \in \mathcal{F}: \ F\cap F\ \neq \emptyset} \mathbb{E} (I_F I_F')$. Then
	\[ \mathbb{P}\left(I_{\mathcal{F}}=0\right)\leq \exp \left(-\frac{\mu^2}{2\Delta}\right).\]
\end{prop}		
\begin{proof}[Proof of Lemma~\ref{lemma: connecting}]
	Fix $C_i$ with $M\leq \vert C_i \vert \leq 2M$. Set $M'=\min(2M, n)$. Let $\mathcal{F}_0$ denote the collection of \emph{connecting triples} $(x_1x_2, y_1y_2, z_1z_2) \in C_i \times \bigcup_{j\neq i}C_j  \times S$. Further let 
	\[\mathcal{F}=\{\{x_iz_j: \ i,j \in [2] \}\cup \{y_iz_j: \ i,j \in [2] \}: \ (x_1x_2, y_1y_2, z_1z_2) \in \mathcal{F}_0\}.\]
	Observe that the elements of $\mathcal{F}$ are subsets of either $6$ 
	or $8$ edges (depending on whether the pairs $x_1x_2$ and $y_1y_2$ overlap or not) of the 
	complete bipartite graph $B(V,W)$ with bipartition $V\sqcup W$. 
	We shall apply Janson's inequality to $\mathcal{F}$ to give an 
	upper bound on the probability that $C_i$ does not connect to $\bigcup_{j\neq i} C_j$ via a pair of squares of $B_p(V,W)$. To this end, we must compute and bound the $\mu$ and $\Delta$ parameters for $\mathcal{F}$. The first of these is straightforward:
	\begin{align}\label{inequality: mu bound}
	\mu:&=\mathbb{E} I_{\mathcal{F}} = \sum_{F \in \mathcal{F}} \mathbb{E} I_F \geq \vert C_i\vert . \bigl\vert \bigcup_{j\neq i} C_j \bigr\vert . \vert S\vert p^8
	\geq M \frac{\varepsilon_2(\varepsilon_1)^2}{8} \frac{n^4}{(\log n)^2}p^8
	\end{align}
	To bound the $\Delta$ parameter, fix a connecting triple $t=(x_1x_2, y_1y_2, z_1z_2)\in C_i \times \bigcup_{j\neq i}C_j  \times S$, and consider the contribution to $\Delta$ made by pairs $(t,t')$ of connecting triples that share at least one edge of $B(V,W)$; call such pairs of connecting triples \emph{dependent}.

	Write $L(t)$ for the set $\{x_1, x_2, y_1, y_2\}$ (which can have size either $4$ or $3$ --- the latter if one of the $x_i$ is equal to one of the $y_j$) and $R(t)$ for the pair $\{z_1,z_2\}$. Also let $F_t\in \mathcal{F}$ be the collection of edges of $B(V,W)$ from $L(t)$ to $R(t)$. Clearly if $L(t)\cap L(t')=\emptyset$ or $R(t)\cap R(t')=\emptyset$, then $(t,t')$ do not form a pair of dependent connecting triples.

Fix a connecting triple $t$. For $(i,j)\in [4]\times [2]$, let $D_{i,j}(t)$ denote the collection of connecting triples $t'$ with $\vert L(t)\cap L(t')\vert = i$ and $\vert R(t)\cap R(t')\vert=j$. Further let $D_{i,j}^a(t)$ and $D_{i,j}^b(t)$ denote the collection of $t'$ in $D_{i,j}(t)$ with $\vert L(t)\vert=4$ and $\vert L(t')\vert=3$ respectively.  We shall bound the sizes of the sets  $D_{i,j}^a(t)$ and $D_{i,j}^b(t)$. Note to begin with that there are at most $\frac{2\varepsilon_1n}{\log n}$ ways of deleting a vertex in $R(t)$ and replacing it by a different vertex in $W$. In particular, for all connecting triples $t$ and all $i\in [\vert L(t)\vert]$, we have
\begin{align*}
\vert D_{i, 1}^a(t)\vert \leq \vert D_{i,2}^a(t)\vert \cdot \frac{2\varepsilon_1n}{\log n} && \textrm{ and }&& \vert D_{i, 1}^b(t)\vert \leq \vert D_{i,2}^b(t)\vert \cdot \frac{2\varepsilon_1n}{\log n}.
\end{align*}
Thus we may focus on bounding the sizes of $D_{i,j}^a(t)$ and $D_{i,j}^b(t)$ in the case where $j=2$.

\noindent \textbf{Case 1, $\vert L(t)\vert=4$:}
	\begin{itemize}
		\item there are at most $6$ ways of splitting $L(t)$ into a
		pair from $C_i$ and a pair from $\bigcup_{j\neq i} C_j$, and
		at most $24$ ways of deleting a vertex from $L(t)$ and viewing
		the remaining $3$ vertices as the union of a pair from $C_i$
		and an (overlapping) pair from $\bigcup_{j\neq i} C_j$, whence $\vert D_{4,2}^a(t)\vert \leq 6$ and $\vert D_{3,2}^b(t)\vert \leq 24$;
		
		\item there are at most $4n$ ways of deleting one vertex from
		$L(t)$ and replacing it by another vertex from $V$. As noted above, there are most $6$ ways of splitting the resulting $4$-set into a pair from $C_i$ and a pair from $\bigcup_{j\neq i} C_j$, whence $\vert D_{3,2}^a(t)\vert \leq 24n$; 
		
		 \item there are at most $6 (n^2/2)=3n^2$ ways of deleting two
		vertices from $L(t)$ and replacing them by two other vertices
		from $V$, whence (similarly to the above) we have $\vert D_{2,2}^a(t)\vert \leq 18n^2$; further, there are at most $6n$ ways of deleting a pair of vertices from $V$ and replacing them by a single vertex from $V$, whence (similarly to the above, since there are most $6$ ways of viewing three vertices of a pair from $C_i$ and an (overlapping) pair from $\bigcup_{j\neq i} C_j$) we get  $\vert D_{2,2}^b(t)\vert \leq 36n$; 
		
		\item since $C_i$ contains at most $2M$ pairs, there are at
		most $4(n^2/2)M'=2n^2M'$ ways of deleting three vertices in $L(t)$ and
		replacing them by another three vertices from $V$ in such a way
		that the resulting set can still be viewed as the union of a
		pair from $C_i$ and a pair from $\bigcup_{j\neq i} C_j$, whence (similarly to the above) we have $\vert D_{1,2}^a(t)\vert\leq 12n^2M'$; further and similarly there are at most $4M+ 4M'n\leq 8M'n$ ways of deleting three vertices in $L(t)$ and replacing them by a pair of vertices from $V$, whence (as before) we have $\vert D_{1,2}^b(t)\vert \leq 48 M'n$; 
			\end{itemize}
\noindent \textbf{Case 2, $\vert L(t)\vert=3$:}
	\begin{itemize}
		
		\item there are at most $6$ ways of splitting $L(t)$ into a pair from from $C_i$ and a pair from $\bigcup_{j\neq i} C_j$, whence $\vert D_{3,2}^b(t)\vert \leq 6$; further there are at most $6n$ ways of adding a vertex to $L(t)$ and splitting the resulting $4$-set into two disjoint pairs, whence $\vert D_{3,2}^a(t)\vert \leq 6n$;
	
		\item  there are at most $3n$ ways of deleting one vertex from $L(t)$ and replacing it by another vertex from $V$, whence (as above) $\vert D_{2,2}^b(t)\vert \leq 18n$; further, there are at most $3(n^2/2)$ ways of deleting one vertex from $L(t)$ and replacing it by a pair from $V$, whence (as above) $\vert D_{2,2}^a(t)\vert \leq 9n^2$;
		
		\item since $C_i$ contains at most $2M$ pairs, there are at most $3M'n$ ways of deleting two vertices in $L(t)$ and replacing them by another two vertices from $V$ in such a way that the resulting set can still be viewed as the union of a pair from $C_i$ and a pair from $\bigcup_{j\neq i} C_j$, whence $\vert D_{1,2}^b(t)\vert \leq 18M'n$; similarly, there are at most $3M'(n^2/2)$ ways of deleting two vertices in $L(t)$ and replacing them by a triple of vertices from $V$ in such a way that the resulting set can be viewed as the disjoint union of a pair from $C_i$ and a pair from $\bigcup_{j\neq i} C_j$, whence $\vert D_{1,2}^a(t)\vert \leq 9M'n^2$. .
	\end{itemize}

	Given $t' \in D_{i,j}^a(t)$ and considering the edges between $L(t)\cup L(t')$ and $R(t)\cup R(t')$, we see that 
	\begin{align}\label{equality: extra edges in (i,j)-a-intersecting triple}
	\mathbb{E}I_{F_t}I_{F_{t'}}= \mathbb{E} I_{F_t} p^{2(4-i) +i(2-j)}=  \mathbb{E} I_{F_t} p^{8-ij}.
	\end{align}
	Similarly, for $t'\in D^b_{i,j}(t)$ we have 
	\begin{align}\label{equality: extra edges in (i,j)-b-intersecting triple}
	\mathbb{E}I_{F_t}I_{F_{t'}}= \mathbb{E} I_{F_t} p^{2(3-i) +i(2-j)}=  \mathbb{E} I_{F_t} p^{6-ij}.
	\end{align}
	With the bounds on the size of $D_{i,j}^a(t)$ and $D_{i,j}^b(t)$ derived above and equalities \eqref{equality: extra edges in (i,j)-a-intersecting triple} and \eqref{equality: extra edges in (i,j)-b-intersecting triple} in hand, we are now ready to bound the contribution to $\Delta$ from a connecting triple $t$.

	\noindent \textbf{Case 1: $\vert L(t)\vert =4$.} 
	\begin{align}
	\sum\Bigl\{ \mathbb{E}I_{F_t}I_{F_{t'}}\Bigr.&\Bigl.: \ (t,t') \textrm{ form a pair of dependent connecting triples}\Bigr\}\notag \\
	=&\ \mathbb{E}I_{F_t} \left( \sum_{i=1}^4 \left(\vert D_{i,2}^a\vert p^{8-2i} + \vert D_{i,2}\vert^b p^{6-2i} \right) + \sum_{i=1}^4 \left(\vert D_{i,1}^a\vert p^{8-i} + \vert D_{i,1}^b\vert p^{6-i} \right)\right)\notag\\
	\leq&\  \mathbb{E}I_{F_t} \Bigl( \left((12n^2M'p^6+ 48M'np^4)+ (18n^2p^4+36np^2)+ (24n p^2 +24) + (6)\right)\Bigr. \notag\\
	&\Bigl.  \  \ +\frac{2\varepsilon_1 n}{\log n}\left((12n^2M'p^7+ 48M'np^5)+ (18n^2p^6+36np^4)+ (24n p^5 +24p^3) + (6p^4)\right)   \Bigr)\notag\\
	\leq&\  \left(\mathbb{E}I_{F_t} \right)2^{10}(np^2)^3 \max\left(1, \frac{\varepsilon_1}{\log n} M'p \right). \label{inequality: bound on contribution from a triple t with L(t) of size 4}
	\end{align}
\noindent (Note in the last line we use the fact that $np^2\geq (\lambda_c)^2>1/4$, whence $(np)^{-1}\leq 4p$.)	
	
\noindent	\textbf{Case 2: $\vert L(t)\vert =3$.} 
	
	\begin{align}
	\sum\Bigl\{ \mathbb{E}I_{F_t}I_{F_{t'}}\Bigr.&\Bigl.: \ (t,t') \textrm{ form a pair of dependent connecting triples}\Bigr\}\notag \\
	=&\ \mathbb{E}I_{F_t} \left( \sum_{i=1}^3 \left(\vert D_{i,2}^a\vert p^{8-2i} + \vert D_{i,2}^b\vert p^{6-2i} \right) + \sum_{i=1}^3 \left(\vert D_{i,1}^a\vert p^{8-i} + \vert D_{i,1}^b\vert p^{6-i} \right)\right)\notag\\
	\leq&\  \mathbb{E}I_{F_t} \Bigl( \left((9n^2M'p^6+ 18nM'p^4)+ (9n^2p^4+18np^2)+ (6n p^2 +6)\right)\Bigr. \notag\\
	&\Bigl.  \  \ +\frac{2\varepsilon_1 n}{\log n}\left((9n^2M'p^7+ 18nM'p^5)+ (9n^2p^6+18np^4)+ (6n p^5 +6p^3)\right)  \Bigr)\notag\\
	\leq&\  \left(\mathbb{E}I_{F_t} \right)2^{10}(np^2)^3 \max\left(1, \frac{\varepsilon_1}{\log n} M'p \right).\label{inequality: bound on contribution from a triple t with L(t) of size 3}
	\end{align}
	
	Together, inequalities~\eqref{inequality: bound on contribution from a triple t with L(t) of size 4} and \eqref{inequality: bound on contribution from a triple t with L(t) of size 3} yield that 
	\begin{align}\label{inequality: Delta bound}
	\Delta \leq \frac{1}{2}\mu. 2^{10}(np^2)^3 \max\left(1, \frac{\varepsilon_1}{\log n} M'p \right).
	\end{align}
	Applying the Extended Janson Inequality, Proposition~\ref{proposition: Janson inequality}, together with the bounds~\eqref{inequality: mu bound} and \eqref{inequality: Delta bound} on $\mu$ and $\Delta$, we get:
	\begin{align}\label{inequality: Janson bound for one component}
	\mathbb{P}(I_{\mathcal{F}}=0)&\leq \exp \left(-\frac{\mu^2}{2\Delta}\right)\leq \exp \left(-\frac{\mu}{2^{10}(np^2)^3 \max\left(1, \frac{\varepsilon_1}{\log n} M'p \right)}\right)\notag\\
	&\leq \exp \left(- \frac{\varepsilon_2(\varepsilon_1)^2}{2^{13}}\frac{M(np^2)}{(\log n)^2 \max\left(1, \frac{\varepsilon_1}{\log n} M'p \right)} \right)\notag \\
	&\leq \left\{ \begin{array}{ll}
	\exp \left( -\frac{\varepsilon_2(\varepsilon_1)^2}{2^{15}} \frac{M}{(\log n)^2}\right) & \textrm{if } 2M\leq \frac{p^{-1} \log n}{\varepsilon_1},\\
	\exp \left( -\frac{\varepsilon_2\varepsilon_1}{2^{15}} \frac{p^{-1}}{\log n}\right) & \textrm{if } \frac{p^{-1} \log n}{\varepsilon_1}\leq 2M \leq n,\\
	\exp \left( -\frac{\varepsilon_2\varepsilon_1}{2^{15}} \frac{p^{-1}M}{n\log n}\right) & \textrm{if }n \leq 2M.
	\end{array} \right.
	\end{align}
Now, the probability that $C_i$ fails to connect to $\cup\{C_j:\ j \neq i\}$ via a connecting triple is exactly the probability that $I_{\mathcal{F}}=0$.	Applying Markov's inequality together with \eqref{inequality: Janson bound for one component} and using our assumptions that  $M\geq (\log n)^4$ and that  $C_1, \ldots, C_r$ are disjoint, we have for $n$ sufficiently large that
	\begin{align*}
	\mathbb{P}(\textrm{Boost})&\geq 1-r \mathbb{P}(I_{\mathcal{F}}=0)\geq 1-\frac{n^2}{M} \mathbb{P}(I_{\mathcal{F}}=0)\geq 1- \exp \left( -\frac{\varepsilon_2(\varepsilon_1)^2}{2^{16}}(\log n)^2 \right),
	\end{align*}
	provided $n$ is sufficiently large. This concludes the proof of the connecting lemma.
\end{proof}	

\subsection{Sprinkling vertices}\label{section: vertex sprinkling}
With Lemma~\ref{lemma: connecting} in hand, we can return to the proof of Theorem~\ref{thm:1giant}. 
To complete the proof, we shall use a multiple-round vertex-sprinkling argument. We partition $\Gamma\in \G(n,p)$ into the union of 
\begin{enumerate}[label=(\roman*)]
	\item $\Gamma_1\in \G((1-\varepsilon_1)n, p)$ on $V_1=[(1-\varepsilon_1)n]$,
	\item $\Gamma_2\in \G(\varepsilon_1n, p)$ on $V_2=[n]\setminus V_1$, and
	\item a $p$-random bipartite graph $B=B_p(V_1, V_2)$ with bipartition $V_1\sqcup V_2$.
\end{enumerate}
We further partition $V_2$ into $2\log_2 n$ sets of size $\frac{\varepsilon_1 n}{2\log_2 n}$, $V_2=\sqcup_{i=1}^{2\log_2 n} V_{2,i}$ (and ignore rounding errors). We say that $\Gamma_1$ is a \emph{good configuration} if it satisfies the conclusion of Corollary~\ref{corollary: lots of non-edges in large-ish components (1-eps_1) version}, i.e., if at least $\varepsilon_2 n^2$ non-edges of $\Gamma_1$ lie in square-components in $\square(\Gamma_1)$ of order at least $M_0:=(\log n)^4$ (this is actually slightly weaker than what Corollary~\ref{corollary: lots of non-edges in large-ish components (1-eps_1) version} gives us, but is all we need here).

We shall condition on $\Gamma_1$ being a good configuration when we 
perform our vertex-sprinkling. By Corollary~\ref{corollary: lots of 
non-edges in large-ish components (1-eps_1) version}, this occurs 
with probability $1-O(n^{-1})$. A key observation is that the state 
of the edges in $\Gamma_2$ and $B$ are independent of our 
conditioning. Our strategy is then to reveal the $2\log_2 n$ 
sprinkling sets $V_{2,k}$ one by one, and use them to create bridges 
between ``somewhat large'' square-components and thereby increase the 
minimum order of all ``somewhat large'' square-components.

More precisely, before stage $k\geq 1$ we have revealed all the edges 
inside \[V_{1,k-1}:=V_1 \cup \left(\cup_{i=1}^{k-1} V_{2,i}\right).\] 
At this stage, we deem a square-component ``large'' if it contains at 
least $M_{k-1}$ non-edges of $\Gamma$, and ``very large'' if it contains at least $2M_{k-1}$ non-edges of $\Gamma$ (which constitute, as we recall, the vertices of the square-graph). Now in stage $k$, we reveal the set $S_k$ of non-edges of $\Gamma$ that lie inside $V_{2, k}$ and the edges between $V_{1,k}$ and $V_{2,k}$. We then merge components as follows: given two square-components $C_i$ and $C_j$, a \emph{connecting triple} is a triple $(x_1x_2, y_1y_2, z_1z_2)\in C_i \times C_j \times S_k$. Such a connecting triple is \emph{active} if all edges between the sets $\{x_1,x_2, y_1, y_2\}$ and $\{z_1,z_2\}$ are in $\Gamma$; in this case the components $C_i$ and $C_j$ lie inside the same square-component $C$ in $\square (\Gamma[V_{1,k}])$ (see Figure~\ref{fig:connecting}). In particular, if both $C_i$ and $C_j$ contained at least $M_{k-1}$ non-edges, then $C$ must contain at least $M_k=2M_{k-1}$ non-edges.

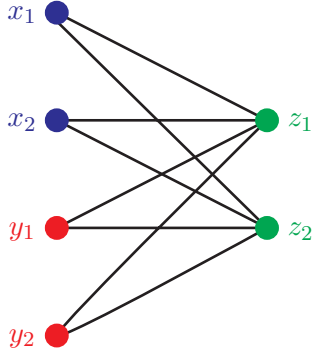
\begin{figure}\centering
	
		\labellist
	\pinlabel {\textcolor{figred}{$y_1$}} [ ] at  0 107
	\pinlabel {\textcolor{figred}{$y_2$}} [ ] at 0 26
	\pinlabel {\textcolor{figblue2}{$x_1$}} [ ] at 0 270
	\pinlabel {\textcolor{figblue2}{$x_2$}} [ ] at 0 189
	\pinlabel {\textcolor{figgreen2}{$z_1$}} [ ] at 210 189
	\pinlabel {\textcolor{figgreen2}{$z_2$}} [ ] at 210 107
	\endlabellist
	\includegraphics[scale=0.5]{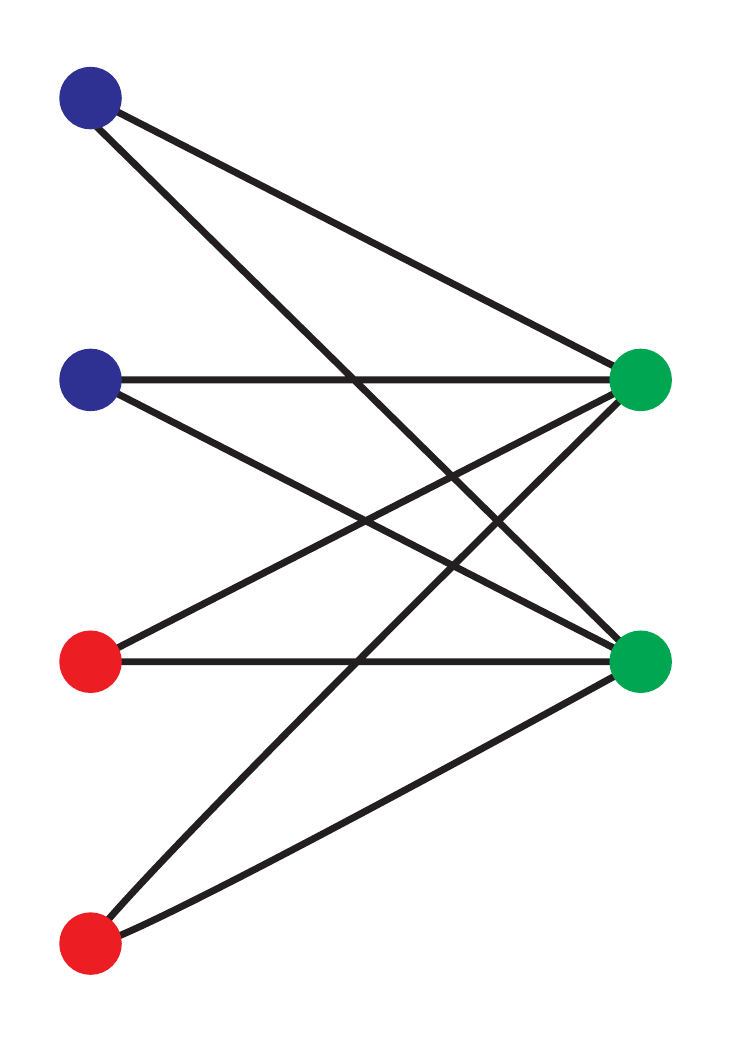}
	\caption{In this connecting triple, both $x_1z_1x_2z_2$ and $y_1z_1y_2z_2$ form induced copies of $C_4$, joining up the square-components containing the non-edges $x_1x_2$ and $y_1y_2$ via the non-edge of sprinkled vertices $z_1z_2$.}\label{fig:connecting}
\end{figure}
The connecting lemma we proved in the previous subsection immediately 
implies that with high probability at each stage $k$, all components 
which are ``large'' but not ``very large'' must join up with at least 
one other ``large'' component.  We make this explicit with a lemma below. Recall that throughout this section, $\lambda>\lambda_c$ is fixed and the edge-probability $p=p(n)$ satisfies $\lambda \frac{1}{\sqrt{n}} < p(n) < 5\frac{\sqrt{\log n}}{\sqrt{n}}$. Let $\varepsilon_1, \varepsilon_2>0$ be the constants whose existence is guaranteed by Corollary~\ref{corollary: lots of non-edges in large-ish components (1-eps_1) version}.

\begin{lem}[Sprinkling lemma]\label{lemma: sprinkling}
	Suppose that before stage $k$, at least $\varepsilon_2 n^2$ non-edges
	of $\Gamma[V_{1,k-1}]$ lie in square-components of order at least $M_{k-1}=
	2^{k-1} M_0$ in $\square (\Gamma[V_{1,k-1}])$.  Suppose $M_{k-1}\leq
	\frac{\varepsilon_2}{4}n^2$.  Then with probability at least
	\[ 1-2\exp \left( -\frac{\varepsilon_2(\varepsilon_1)^2}{2^{16}}(\log n)^2 \right)\] 
	when we have revealed the edges from $V_{2,k}$ to $V_{1,k-1}\cup
	V_{2,k}:=V_{1,k}$ at least $\varepsilon_2 n^2$ non-edges of $\Gamma[V_{1,
		k}]$ lie in square-components of order at least $M_{k}= 2M_{k-1}$
	in $\square (\Gamma[V_{1,k}])$.

	In particular, with probability at least
	\[1- 4\log_2 n\exp \left( -\frac{\varepsilon_2(\varepsilon_1)^2}{2^{16}}(\log n)^2 \right)  \]
	we have that starting from a good configuration $\Gamma[V_1]$, the sprinkling process described above has discovered within $2\log_2 n - 4\log_2 \log n$ steps a giant square-component containing at least $\varepsilon_2n^2$ non-edges, and all non-edges from $\Gamma[V_1]$ that lie inside components of $\square(\Gamma[V_1])$ of size at least $(\log n)^4.$
\end{lem}
\begin{proof}
	Let $S_k$ denote the set of non-edges of $V_{2,k}$. We have $\mathbb{E}\vert S_k\vert =(1-p) \binom{\frac{\varepsilon_1}{\log n}n}{2}= (1-o(1)) \frac{\varepsilon_1^2}{(\log n)^2}n^2$ By a standard Chernoff bound,
	\begin{align*}
	\mathbb{P}\left(\vert S_k\vert \leq \frac{\varepsilon_1n^2}{4(\log n)^2}\right)  &\leq \exp\left(-\frac{(\varepsilon_1)^2}{16}\frac{n^2}{(\log n)^2}\right) \end{align*}
	If $\vert S_k\vert \geq \frac{\varepsilon_1n^2}{(\log n)^2}$ holds, we can apply Lemma~\ref{lemma: connecting}, concluding that every component of size $M_{k-1}$ is joined with at least one other, resulting in a component of size $M_k=2M_{k-1}$. Thus the desired conclusion for the first part of the lemma holds with probability at least 
	\[  1- \exp\left(-\frac{(\varepsilon_1)^2}{16}\frac{n^2}{(\log n)^2}\right) -\exp \left( -\frac{\varepsilon_2(\varepsilon_1)^2}{2^{16}}(\log n)^2 \right)\geq 1- 2 \exp \left( -\frac{\varepsilon_2(\varepsilon_1)^2}{2^{16}}(\log n)^2 \right),\]
	as desired.

	For the ``in particular'' part,  we first apply a simple union bound to the first $2\log_2 n -4\log_2\log n$ steps of the process, to show that with probability at least 
	\begin{align}\label{inequality: bound on probability of having found large square components}
	1- 2(2\log_2 n -4\log_2\log n )\exp \left( -\frac{\varepsilon_2(\varepsilon_1)^2}{2^{16}}(\log n)^2 \right)
	\end{align}
	our sprinkling process has uncovered a collection of square-components, each of which contains at least $\frac{\varepsilon_2n^2}{2}$ non-edges, and, whose union contains at least $\varepsilon_2n^2$ non-edges and includes all non-edges of $\Gamma[V_1]$ coming from components of $\square(\Gamma[V_1])$ of size at least $(\log n)^4$.  There can be at most $\frac{1}{2}(\varepsilon_2)^{-1}$ such components. By \eqref{inequality: Janson bound for one component}, the probability that a fixed pair of such components fails to join up in the next round of sprinkling is at most
	\[\exp\left(-\frac{(\varepsilon_2)^2\varepsilon_1}{2^{16}}\frac{p^{-1}n}{\log n}\right)\leq  \exp\left(-\frac{(\varepsilon_2)^2\varepsilon_1}{2^{16}}\frac{5n^{\frac{3}{2}}}{(\log n)^\frac{3}{2}} \right). \]
	Taking the union bound over the at most
	$\frac{1}{8}(\varepsilon_2)^{-2}$ pairs of components, we have that
	the probability any pair of these components fail to join up is, for
	large $n$, a lot less than the last term in 
	equation~\eqref{inequality: bound on probability of having found 
		large square components}: 
	\[8\log_2\log n \exp \left( -\frac{\varepsilon_2(\varepsilon_1)^2}{2^{16}}(\log n)^2 \right).\]
	Combining this with \eqref{inequality: bound on probability of having found large square components}, we get the claimed bound on the probability of having discovered a giant square-component containing at least $\varepsilon_2n^2$ non-edges and all non-edges of $\Gamma[V_1]$ contained in components of $\square(\Gamma[V_1])$ of size at least $(\log n)^4$. 
\end{proof}

\subsection{Covering the whole world} 
All that now remains to complete the proof of 
Theorem~\ref{thm:1giant} is to show that a.a.s.\ there is a 
square-component covering all vertices of $\square(\Gamma)$. Note 
that while in Lemma~\ref{lemma: sprinkling} we showed that 
$\square(\Gamma)$ has a giant component, we have not quite shown it 
is unique: in principle, one could stitch together a rival giant 
component at the last stage of sprinkling by building numerous 
bridges between small components. This is of course highly unlikely, 
and one could show uniqueness of the giant by exploiting the fact 
that the number of non-edges of $\Gamma$ lying in square-components of order at least $(\log n)^4$ is a.a.s.\ concentrated around its expectation (as shown in Corollary~\ref{corollary: lots of non-edges in large-ish components (1-eps_1) version}). However we do not have a nice form for this expectation, so a little care would be needed to show it changes continuously with $n$ to make the argument above fully rigorous. As this paper is already sufficiently long and as the uniqueness of the giant is not our main concern here, we eschew this and focus instead on the problem of ensuring we have a giant component whose support covers all the vertices. We sidestep the issue of the uniqueness of the giant by considering a partition of $[n]$ which allows us to both build a preferred giant and, crucially, to ensure this preferred giant has full support. We begin by establishing a useful corollary of the work in the previous subsections. 

\begin{cor}\label{corollary: existence of square-giants}
	Let $\lambda>\lambda_c$ be fixed, and let $p=p(n)$ be an 
	edge-probability satisfying $\lambda n^{-1/2}\leq p\leq 5 
	n^{-1/2} \sqrt{\log n}$. Let $\Gamma\in \G(n,p)$. Then there for every	$\varepsilon_3>0$ sufficiently small, there exists a constant $\varepsilon_4>0$ such that given fixed sets $U\subseteq U'\subseteq [n]$ with $\vert U\vert =\lfloor (1-2\varepsilon_3)n\rfloor$, $\vert U'\vert  = \lfloor(1-\varepsilon_3)n\rfloor$ all of the following hold with probability $1-O(n^{-1})$:
	\begin{enumerate}[label=(\roman*)]
		\item\label{cor:i} there are at least $\varepsilon_4n^2$ non-edges in $\Gamma[U]$ contained in square-components of $\square(\Gamma[U])$ of order at least $(\log n)^4$;
		\item\label{cor:ii} there is a unique square-component in $\square(\Gamma[U'])$ containing all non-edges in $\Gamma[U]$ contained in square-components of $\square(\Gamma[U])$ of order at least $(\log n)^4$;
		\item\label{cor:iii} there is a unique square-component in $\square(\Gamma)$ containing all non-edges contained in square-components of $\square(\Gamma[U])$  or $\square(\Gamma[U'])$ of order at least $(\log n)^4$.\qed
		\end{enumerate}	
\end{cor}

\begin{proof} 
	The corollary is immediate for sufficiently small $\varepsilon_3>0$ from an application of Corollary~\ref{corollary: lots of non-edges in large-ish components (1-eps_1) version} inside $U$ (for part (i)) and two applications of Lemma~\ref{lemma: sprinkling} (for parts (ii) and (iii) respectively), together with a suitable choice of the constant $\varepsilon_4$.
\end{proof}
\noindent We now apply this Corollary to prove Theorem~\ref{thm:1giant}.

\begin{proof}[Proof of Theorem~\ref{thm:1giant}]
	
First note that if $f(n)$ is any function with $f(n)=o(1)$ and $f(n)=\Omega(n^{-2})$,  and $5n^{-1/2}\sqrt{\log{n}}\le p(n)\le 1-f(n)$, then $\Gamma\in\G(n,p)$ has the $\CFS$ property, by \cite[Theorem 5.1]{BehrstockFalgasRavryHagenSusse:StructureRandom}.  Now assume that $\lambda>\lambda_c$ and $\lambda n^{-1/2}\le p(n)\le 5n^{-1/2}\sqrt{\log{n}}$. Pick $\varepsilon_3>0$ sufficiently small, and let $\varepsilon_4>0$ be the constant whose existence is guaranteed by Corollary~\ref{corollary: existence of square-giants}. Partition $[n]$ into $K=\lceil 2(\varepsilon_3)^{-1}\rceil $ sets 
\[U_i= \{x\in [n]: \     (i-1)\varepsilon_3n<x\leq i \varepsilon_3 n\}.\]
For each pair $(i,j)$ of distinct elements of $[K]$, we apply  Corollary~\ref{corollary: existence of square-giants}  to the sets $U=[n]\setminus (U_i\cup U_j)$ and $U'=[n]\setminus U_i$; taking a union bound over all such pairs $(i,j)$, we  see that with probability $1-O(K^2n^{-1})=1-O(n^{-1})$, for every pair of distinct elements $i,j\in [K]$ the following hold:
\begin{enumerate}[label=\textbf{(\arabic*)}]
	\item at least $\varepsilon_4n^2$ non-edges in $\Gamma[[n]\setminus (U_i\cup U_j)]$ are contained in components of $\square(\Gamma[[n]\setminus (U_i\cup U_j)])$ of order at least $(\log n)^4$;
	\item there is a unique component $C'_{ij}$ of $\square(\Gamma[[n]\setminus U_i])$ containing all non-edges of $\Gamma[[n]\setminus(U_i\cup U_j)]$ contained in components of $\square(\Gamma[[n]\setminus (U_i\cup U_j)])$ of order at least $(\log n)^4$;
	\item there is a unique component $C_{i}$ of $\square(\Gamma)$ containing $C'_{ij}$ as well as all non-edges of $\Gamma[[n]\setminus U_i]$ contained in components of $\square(\Gamma[[n]\setminus U_i])$ of order at least $(\log n)^4$.
\end{enumerate} 
We claim that $\forall i,j \in [K]$ we have $C_i=C_j$. Indeed for $i\neq j$, note that $C_i \supseteq C'_{ij}$ and $C_j\supseteq C'_{ji}$. Since both $C'_{ij}$ and $C'_{ji}$ contain all of the at least $\varepsilon_4n^2$ non-edges contained in components of $\square(\Gamma[[n]\setminus (U_i\cup U_j)])$ of order at least $(\log n)^4$, it follows that $C_i \cap C_j \supseteq C'_{ij}\supseteq C'_{ji}\neq \emptyset$. Since their intersection is non-empty, $C_i$ and $C_j$ are the same component of $\square(\Gamma)$, as claimed. We may thus let $C_{\star}$ denote the a.a.s.\ unique square-component with $C_{\star}=C_i$ for all $i\in [K]$.

We now show that a.a.s.\ the support of this component $C_{\star}$ is the whole vertex set $[n]$. Pick $i\in[K]$ and condition on the event that  there is a square-component $C_i'$ in $\square(\Gamma[[n]\setminus U_i])$ of order at least $\varepsilon_4n^2$ (an event which occurs with probability $1-O(n^{-1})$, as we saw in \textbf{(2)} above). If two or more such components exist, pick a largest one. Further, condition on each vertex $x\in U_i$ having at least $\frac{\varepsilon_3}{2} n$ non-neighbors in $\Gamma[U_i]$. By a standard application of Chernoff bounds and a union-bound, this event occurs with probability $1-O(n^{-1})$.

Having thus conditioned on the state of pairs in $\Gamma[U_i]$ and $\Gamma[[n]\setminus U_i]$, we now show that a.a.s.\ for every vertex $x \in U_i$, there exist $y\in U_i$ and $uv\in C_i'$ such that $xyuv$ induces a copy of $C_4$ --- so that that $xy$ belongs to the component $C'$ of $\square(\Gamma)$ containing $C_i'$.  Combining this with \textbf{(3)} above (which implie that a.a.s.\ $C'=C_{\star}$) and a simple union bound will then yield Theorem~\ref{thm:1giant}.

 Given non-edges $xy\in U_i^{(2)}\setminus E(\Gamma[U_i])$ and $uv\in C_i'$, let $X_{xy, uv}$ be the indicator function of the event that all of $ux$, $uy$, $vx$ and $vy$ are edges of $\Gamma$. Observe that this event is independent of our conditioning. For $x\in U_i$, set 
 \[X_x=\sum_{y\in U_i: \ xy \notin E[U_i]}\sum_{uv\in C_i'} X_{xy,uv}\]
 to be the number of induced $C_4$'s $xyuv$ of $\Gamma$  with $y\in U_i$ and $uv\in C_i'$.  We shall again apply the Extended Janson Inequality ((Proposition~\ref{proposition: Janson inequality}) to bound $\mathbb{P}(X_x=0)$.  Given our conditioning, the expectation of $X_x$ is
\[\mu:=\vert \{y\in U_i: \ xy \notin E(\Gamma)\}\vert \cdot \vert C_i'\vert p^4\geq\frac{ \varepsilon_3\varepsilon_4 (\lambda_c)^4}{2} n=\Omega(n).\]
We now compute the corresponding $\Delta$-parameter in Janson's inequality. For $y,y'\in U_i\setminus \Gamma_x[U_i]$ and $uv, u'v'\in C'_i$, write $\{xy, uv\}\sim\{xy',u'v'\}$ if $\{ux, uy, vx, vy\}\cap\{u'x,u'y',v'x,v'y\}\neq\emptyset$. Note that the random variables $X_{xy,uv}$ and $X_{xy', u'v'}$ are independent unless $\{xy, uv\}\sim\{xy',u'v'\}$, and further that $\{xy, uv\}\sim\{xy',u'v'\}$ implies $\{u,v\}\cap \{u',v'\}\neq \emptyset$.

Pick and fix $y\in U_i\setminus \Gamma_x[U_i]$ and $uv\in C'_i$, which can be done in at most $\vert U_i\vert \cdot \vert C_i'\vert $ ways.  We perform a case analysis to determine the contributions to the $\Delta$-parameter from terms of the form $\E X_{xy,uv}X_{xy',u'v'}$ with $\{xy, uv\}\sim\{xy',u'v'\}$:
\begin{itemize}
	\item there are at most $\vert U_i\vert $ choices of $y'\in U_i\setminus \Gamma_x[U_i]$ with $y'\neq y$ such that $\{xy, uv\}\sim\{xy',uv\}$, and for such $y'$ we have 
	$\E[X_{xy,uv}X_{xy',uv}]=p^6$;
	\item there are at most $n$ choices of $v'\in [n]\setminus U_i$ with $v'\neq v$ such that $\{xy, uv\}\sim\{xy,uv'\}$, and for such $v'$ we have 
	$\E[X_{xy,uv}X_{xy,uv'}]=p^6$  (with the contribution from the symmetric cases $\{xy, uv\}\sim \{xy,u'v\}$ analysed similarly);
	\item finally, there are at most $n\vert U_i\vert $ choices of $(v',y')$ with $v'\in [n]\setminus U_i$, $y'\in U_i\setminus \Gamma_x[U_i]$,   $v'\neq v$ and $y\neq y'$ such that $\{xy, uv\}\sim\{xy',uv'\}$, and for such $(v', y')$ we have 
	$\E[X_{xy,uv}X_{xy',uv'}]=p^7$ (with the contribution from the symmetric cases $\{xy, uv\}\sim \{xy',u'v\}$ analysed similarly).
\end{itemize}
Putting it all together, we see  
\begin{align*}
\Delta = \sum_{\{xy, uv\}\sim\{xy',u'v'\}} \E[X_{xy,uv}X_{xy',uv}] \leq  \vert U_i\vert \cdot \vert C_i'\vert \Bigl(\vert U_i\vert p^6 + 2np^6+ 2n\vert U_i\vert p^7  \Bigr).
\end{align*}	
Since $\vert U_i\vert \leq \varepsilon_3 n$,  $\vert C'_i\vert \leq n^2/2$ and $p\leq 5\sqrt{\frac{\log{n}}{n}}$, the above implies
\begin{align*}
\Delta&\le 5^7(\varepsilon_3)^2n\sqrt{n} (\log n)^{7/2} (1+o(1)),
\end{align*}
which for $n$ large enough is at most $\ds\frac{2^{20}\epsilon_3}{\epsilon_4 (\lambda_c)^4} \mu n^{1/2}(\log n)^{7/2} $.
Applying Janson's inequality,
\begin{align*}
\mathbb{P}(X_x=0)\leq \exp \left(- \frac{\mu^2}{2\Delta} \right)\leq \exp \left(- \frac{\epsilon_4\lambda_c^4\mu}{2^{20}(\log n)^{3/2}\sqrt{n}} \right)=o(n^{-1}).
\end{align*}
Thus with probability $1-o(n^{-1})$, $X_x>0$ and the vertex $x$ is
covered by the square-component in $\square(\Gamma)$ that contains $C_i'$. Taking a union bound over $x\in U_i$, with probability $1-o(1)$, every vertex $x\in U_i$ is
covered by this square-component (which as we showed is the a.a.s\  uniquely determined giant component $C_{\star}$).  Taking a union bound over $i\in [K]$ and combining this with \textbf{(1)--(3)} above, we see that with probability $1-o(1)$ the square-component $C_{\star}$ covers all of $[n]$. This concludes the proof of Theorem~\ref{thm:1giant}.
\end{proof}

\section{Concluding remarks}\label{section: conclusion}
There are several natural questions arising from our work.  To begin
with, one could ask for more information about the component structure
in $\square(\Gamma)$: in the subcritical regime, can one get a better
upper bound on the order of square-components?  In the supercritical
regime, can one give good bounds on the order of the second-largest
square-component?  In particular, can one give better bounds than just
$o(n^2)$, and can one show its support has size $o(n)$?  This may be
feasible albeit technically challenging.

Another question on the probabilistic side is determining the range of
$p=p(n)$ for which the square graph $\square(\Gamma)$ of $\Gamma\in \G(n,p)$ is a.a.s.\ connected, a very different question from the ones considered in this paper. Investigating other parameters such as the diameter of 
$\square(\Gamma)$ may also lead in interesting directions.

Further afield, one could consider percolation problems for similar structures. We could for example consider the graph on all triples of independent vertices in $\Gamma$ obtained by setting two such triples to be adjacent if they induce a copy of the $6$-cycle $C_6$ or of the complete balanced bipartite graph $K_{3,3}$ in $\Gamma$. We would guess the techniques in this paper would adapt well to the latter problem, as Bollob\'as and Riordan also suggest, but that new ideas would be needed for the former.

One could also consider a similar problem for a $p$-random
$r$-uniform hypergraph $\Gamma$, by setting a set of $r$ vertex-disjoint non-edges $F_1,F_2, \ldots, F_r \notin E(\Gamma)$ to be adjacent if all of the $r^2$ edges meeting each $F_i$ in exactly one vertex are present in $\Gamma$ (in other words, if and only if the union of the $F_i$s induces a copy of the complete balanced $r$-partite $r$-uniform hypergraph on $r^2$ vertices). Note the case $r=2$ corresponds exactly to square percolation. Levcovitz~\cite{Levcovitz:QIinvariantThickness} has provided a quasi-isometry invariant for right-angled Coxeter groups by associating a hypergraph to any such group, so analysis of  suitable variants of square percolation for hypergraphs may  yield interesting applications in geometric group theory (besides constituting a challenging and rather natural problem in combinatorial probability).

Finally it would be interesting to study other properties of the
right-angled Coxeter group $W_{\Gamma}$ when $\Gamma\in \G(n,p)$ using
tools from random graph theory.  In particular, determining the
threshold for algebraic thickness of every order, or the exact rate of
divergence of $W_{\Gamma}$ for all $p$ would be of great interest (see
\cite[Question~1]{BehrstockHagenSisto:coxeter}).  Doing so will
require new group theoretic ideas to translate these properties into
graph theoretic language, and the identification of suitably tractable
graph theoretic proxies for these in $\G(n,p)$.  Work of
Levcovitz~\cite{Levcovitz:QIinvariantThickness} provides promising
progress towards finding combinatorial 
properties to encode higher rates of 
polynomial divergence in right-angled Coxeter groups; indeed, as we finalized this paper, Levcovitz
released a new preprint \cite{Levcovitz:DivergenceRACG} that provides
such a translation, which we expect will be of use in future work on this
problem. As Levcovitz's work involves hypergraphs, developing new techniques for 
generalizations of square percolation to hypergraphs will likely be key to further progress.

Finally, one could study thickness and relative hyperbolicity in
random right-angled Coxeter groups with presentation graphs drawn from
other distributions than the Erd{\H o}s--R\'enyi random graph model,
such as random regular graphs. We do not know of any work which has 
been done in this direction at the present time.

\bibliographystyle{acm}
\bibliography{thick_cox}

\end{document}